\documentclass[a4paper,oneside, 11pt]{amsart}
\usepackage{amssymb}
\usepackage{amsmath}
\usepackage{amsfonts}
\usepackage{amsthm}
\usepackage{subfigure}
\usepackage{epsfig}
\usepackage[all]{xy}
\usepackage[utf8]{inputenc}
\usepackage{mathabx}
\usepackage{graphicx}
\usepackage{xcolor}
\usepackage[T1]{fontenc}
\usepackage{comment}
\usepackage{mathtools}
\usepackage{hyperref} 
\hypersetup{colorlinks=true,linkcolor=blue,citecolor=red,urlcolor=blue}

\newtheorem{thm}{Theorem}[section]
 
 \newtheorem{lem}[thm]{Lemma}
 \newtheorem{prop}[thm]{Proposition}
 \newtheorem{theoremname}{Theorem S}
 
 \theoremstyle{definition}
 \newtheorem{defn}[thm]{Definition}
 \theoremstyle{remark}
 \newtheorem{rem}[thm]{Remark}
 \newtheorem{ex}[thm]{Example}
 
 \numberwithin{equation}{section}

\usepackage{xcolor}
\usepackage[all]{xy}
\newcommand{\V}{\mathbf{v}}

\newcommand\sbmattrix[4]{\textnormal{\scriptsize$\left(\begin{array}{cc}#1&#2\\#3&#4\end{array}\right)$\normalsize}}

\title[Diophantine Approximation in local function fields]
{Diophantine Approximation in local function fields via Bruhat-Tit trees}

\begin{document}

%
%
%
%
%
%
%
%
%

\maketitle

\author{\textbf{Luis Arenas-Carmona}
\footnote{Universidad de Chile, Facultad de Ciencias, Casilla 652, Santiago, Chile. Email: \email{learenas@u.uchile.cl}.}}
\textbf{and}
\author{\textbf{Claudio Bravo}
\footnote{Centre de Mathématiques Laurent Schwartz, École Polytechnique, Institut Polytechnique de Paris, 91128 Palaiseau Cedex, France. Email: \email{claudio.bravo-castillo@polytechnique.edu}.}
}

\vspace{0.5cm}

\begin{abstract}
We use the theory of arithmetic quotients of the Bruhat-Tits tree
developed by Serre and others to obtain Dirichlet-style theorems for
Diophantine approximation on global function fields. This approach allows 
us to find sharp values for the constants involved and, occasionally,
explicit examples of badly approximable quadratic irrationals. 
Additionally, we can use this method to easily compute the measure 
of the set of elements that can be written as the limit of a sequence 
of ``better than expected'' approximants. All these results can be
easily obtained via continued fractions when they are available,
so that quotient graphs can be seen as a partial replacement of
them when this fails to be the case.
\\


\end{abstract}


\section{Introduction}\label{section intro}

Let $K$ be the global function field of a smooth irreducible curve $X$ over a finite field $\mathbb{F}=\mathbb{F}_q$.
Let $K_\infty$ the completion of $K$ corresponding to a closed point $P_\infty$ in $X$, which we refer to, in all that follows, as the point at infinity. Let 
$\mathcal{O}_X$ be the structure sheaf of $X$, 
$U_{\infty}=X\smallsetminus\{P_\infty\}$ the finite part of the curve,
and $A=\mathcal{O}_X(U_{\infty})$ the ring of functions that are regular 
at all closed points $P\in U_\infty$, the finite points.
In the sequel, we denote by $\nu_{\infty}: K_{\infty} \to \mathbb{Z} \cup \lbrace \infty \rbrace$ the canonical discrete valuation.
The absolute value $\lvert \cdot \rvert_{\infty} : K_{\infty} \to \mathbb{R}_{\geq 0}$ induced by $\nu_{\infty}$ is normalized by $\lvert \alpha \rvert_{\infty} = q^{-\nu_{\infty}(\alpha)}$, for all $\alpha \in K_{\infty}$.
We adopt several conventions regarding $0\in K$ than can be summarized as follows:
$|0|_\infty=q^{-\nu_\infty(0)}=q^{-\infty}=0$.
The ring of integers of $K_{\infty}$ is $\mathcal{O}_{\infty}:=\lbrace \alpha \in K_{\infty} | \, \, \lvert \alpha \rvert_{\infty} \leq 1 \rbrace$.
We fix, once and for all, a uniformizing parameter $\pi \in \mathcal{O}_{\infty}$.

Let $\alpha\in K_\infty$ be any element.
A rational function $\frac fg$ with $f$ and $g$ in $A$ is 
called an $M$-approximation of $\alpha$ when the following condition holds:
\begin{itemize}
\item The ideals $(f)=fA$ and $(g)=gA$ are co-maximal in $A$.
\item $\left|\alpha-\frac fg\right|_\infty<\frac{|\pi|_{\infty}^M}{|g|_\infty^2}$.
\end{itemize}
We say that $\frac fg$ is a good approximation of $\alpha$ if it is an $N$-approximation, where  $N=\left\lceil\frac{2g(X)-1}{\deg(P_\infty)}\right\rceil$ is the least integer not smaller than $\frac{2g(X)-1}{\deg(P_\infty)} \in \mathbb{R}$.
Note that $N=0$ when $X$ is the projective line, so this generalizes the usual notion.
The question on why this is the natural
generalization is answered in detail in \S~\ref{section general approximation}.
For now, we just mention that we have
 a ``nice'' characterization of good approximations in terms of
the quotients graphs studied in \cite[Ch. II]{SerreTrees}.

An element $\alpha\in K_\infty$ is said to be $M$-approximable if there is a sequence of $M$-
approximations of $\alpha$ converging to $\alpha$. When we can take $M=N$, we say $\alpha$ is
well approximable. We denote by $\Psi$ the set of well approximable elements in $K_\infty$,
while $\Psi_M$ denotes the set of $M$-approximable elements, for any integer $M$.

It is proved in \cite[Theorem 5.3]{Baier2} the existence of certain $M'>0$ such that $\Phi_{M'}=K_{\infty}$.
Important evidence for this was also contributed by Ganguly and Ghosh in \cite[Theorem 1.1]{Ganguly} and Baier and Molla in \cite[Theorem 2]{Baier}.
When $A=\mathbb{F}[x]$ is a polynomial ring, good 
approximations can explicitly be described via continued fractions.
To be precise, every element $\alpha\in K_\infty$ can be writen as a continued
fraction of the form
\begin{equation}\label{eqcf1}
\alpha=f_0+\frac1{f_1+\frac1{f_2+\frac1{f_3+\dots}}},
\end{equation}
where every coefficient $f_i$ is a polynomial, and $\deg(f_i)>0$ for $i>0$.
Good approximations of the element $\alpha$ can 
be explicitly constructed by
truncation of the expression in \eqref{eqcf1}.
Example~\ref{E2} proves that the preceding fact does not hold for more general rings $A$.
However, the method of quotient graphs, 
which we develop in this
work, plays a similar role for a more general 
ring $A$. 
To make this statement precise, recall that the 
arithmetic group $\Gamma=\mathrm{GL}_2(A)$ acts on the Bruhat-Tits tree (building) $\mathfrak{t}$ of $\mathrm{SL}_2(K_{\infty})$.
This action defines a quotient graph $\Gamma \backslash \mathfrak{t}$, which we call the $\mathbf{S}$-graph of $\Gamma$. This graph is obtained by attaching
a finite number of rays, called cusps, to a finite graph called the center.
See \S~\ref{subsection quotients} for details.
The connection between quotient graphs and Diophantine approximation 
is to be expected,
since F. Paulin noted that the degrees of the coefficients of the continued fraction \eqref{eqcf1} can be described in terms of a 
``walk'' in the quotient graph for the case $A=\mathbb{F}[x]$, which is a ray. See \cite{Paulin} for details.

Our method provide simple proofs of several results that one naturally
obtains by using continued fraction,
whenever they are available.
The first result of this type
that we present here
 shows that, even when $M \geq M'$,  the set of $M$-approximable elements in $\mathcal{O}_{\infty}$ is large in a measure theoretical sens.

\begin{thm}\label{T1}
The set $\Psi_M\cap\mathcal{O}_\infty$ has full Haar measure for every integer $M$. Furthermore, there
exists a set $\Sigma$ of full Haar measure such that every element $\alpha\in\Sigma$ is the limit of
a sequence $\left\{\frac{f_n}{g_n}\right\}_n$ such that
$$\lim_{n\rightarrow\infty}\left|\alpha-\frac {f_n}{g_n}\right|_\infty|g_n|_\infty^2=0.$$
\end{thm}
Our second result shows that the exponent
$2$ above is actually optimal:
\begin{thm}\label{T1b}
Let $\rho>2$, and let $\Phi_{M,\rho}$ the set of elements $\alpha\in\mathcal{O}_\infty$ that can be 
written as $\alpha=\lim_{n\rightarrow\infty}\frac{f_n}{g_n}$ with 
\begin{equation}\label{eqT1b}
\left|\alpha-\frac {f_n}{g_n}\right|_\infty<\frac{|\pi|_{\infty}^M}{|g_n|_\infty^\rho},
\end{equation}
then $\Phi_{M,\rho}$ has null Haar measure.
\end{thm}
Both of the preceding results are obtained by studying a Markov Chain
associated to the $\mathbf{S}$-graph of $\Gamma$.

Even though the value 
$N=\left\lceil\frac{2g(X)-1}{\deg(P_\infty)}\right\rceil$
seem to play no role in the preceding result, it plays a significant role
in the proofs, particularly in the graph-theoretical interpretation
of approximation. 
It is an upper bound on the set of values
$M$ for which the set $\Psi_M$ is the whole field, as Theorem~\ref{T2} shows:
\begin{thm}\label{T2}
For any ring $A$ we have $\mathcal{O}_{\infty}\not\subset \Psi_{N+1}$.
Moreover, if $\mathcal{O}_\infty\subseteq \Psi=\Psi_N$, then $A$ is isomorphic to the polynomial ring $\mathbb{F}[x]$.
\end{thm}
Note that the condition $A \cong \mathbb{F}[x]$ 
for the exceptional case implies
$K=\mathbb{F}(x)$, and also
$\mathcal{C}=\mathbb{P}^1_{\mathbb{F}}$ and $\deg(P_\infty)=1$. This is the case
treated  in \cite{Paulin}.

Next result shows that
  the optimal constant $M''$ satisfying $\mathcal{O}_\infty\subseteq\Psi_{M''}$ can be explicitly computed
   from $N$ and the combinatorial structure of the $\mathbf{S}$-graph of $\Gamma$. To make this precise, for every pair
  of neighboring vertices $(v,w)$ in $\Gamma\backslash\mathfrak{t}$ we 
  denote by $m_{v,w}$ the number of neighbors of a pre-image of $v$ in $\mathfrak{t}$ that are pre-images of $w$. This number is independent
  of the choice of a pre-image. The numbers $m_{v,w}$ are well known
  when $v$ and $w$ are located in a cusp, but they need to be computed
  separately for vertices in the center. See \S~\ref{section walk} and \S~\ref{section examples} for details.

\begin{thm}\label{T3}
The largest constant  $M''$ such that 
$\mathcal{O}_\infty\subseteq\Psi_{M''}$ has
the form $M''=N-1-\kappa$, where $\kappa$ is bounded by the diameter
of the center of the $\mathbf{S}$-graph and can be explicitly computed 
from the $\mathbf{S}$-graph of $\Gamma$ and the numbers $m_{v,w}$.
\end{thm}

In \S~\ref{section examples} some examples of
$M''$-approximable elements, for this optimal $M''$,
are described as the common point $\alpha\in K_\infty$ in a decreasing sequence of closed balls that 
we make precise in terms of suitable walks in the 
Bruhat-Tits tree. We give a precise construction of these walks 
for points $P_\infty$ of degree $2$ or $3$ when 
$X=\mathbb{P}^1_{\mathbb{F}}$, and  for one example where $X$ 
is a projective elliptic curve. This provides us with an explicit
 method to compute $\kappa$ from the given data $\big((m_{v,w})_{v,w}, \Gamma \backslash \mathfrak{t}\big)$.
 In some cases we can actually described $\alpha$
 as an explicit quadratic irrational.

\section{Conventions on graphs}\label{section graphs}

We start by recalling some basic definitions on graph theory.
We define a graph $\mathfrak{g}$ as a $5$-tuplet $\mathfrak{g}=(V,E,s,t,r)$ satisfying the following statements:
\begin{itemize}
\item[(1)] $V=V_\mathfrak{g}$ and $E=E_\mathfrak{g}$ are sets, respectively called the vertex set and the edge set.
\item[(2)] The three last symbols denote functions. They are the source  $s=s_{\mathfrak{g}} :E\rightarrow V$, 
the target $t=t_{\mathfrak{g}} :E\rightarrow V$ and the reverse $r=r_{\mathfrak{g}} :E\rightarrow E$, and satisfy the identities 
$r(a)\neq a$,\ $ r\big(r(a)\big)=a$ and $s\big(r(a)\big)=t(a)$, for each edge $a$.
\end{itemize}
Our definition is chosen in a way that allows the existence of multiple edges and loops.
Graphs are the objects in a category $\mathfrak{Graphs}$ whose morphisms are simplicial maps
$\gamma:\mathfrak{g}\rightarrow \mathfrak{g}'$, i.e., pair of functions $\gamma_V: V_\mathfrak{g} \rightarrow V_{\mathfrak{g}'}$ and $\gamma_E: \mathrm{E}_\mathfrak{g} \rightarrow \mathrm{E}_{\mathfrak{g}'}$ preserving the source, target and reverse functions.
Group actions are defined analogously.
A group action of $G$ on a graph $\mathfrak{g}$ has no inversions if $g \cdot a\neq r(a)$, for every edge $a$ and every element $g\in G$.
An action of $G$ on $\mathfrak{g}$ without inversions defines a quotient graph\footnote{When the action of $G$ on $\mathfrak{g}$ has inversions, we can replace $\mathfrak{g}$ by its first 
barycentric subdivision in order to obtain a graph where $G$ acts without inversions. 
Thus, the quotient graph can also be defined in that case.} $\mathfrak{g}_G:=G \backslash \mathfrak{g}$ by setting $V_{\mathfrak{g}_G} =G \backslash V_\mathfrak{g}$, $E_{\mathfrak{g}_G}=G \backslash E_\mathfrak{g}$ and defining $s_{\mathfrak{g}_G}, t_{\mathfrak{g}_G}$ and $r_{\mathfrak{g}_G}$ in a way that turns the canonical projection into a simplicial function.

Two vertices $v,w$ in a graph $\mathfrak{g}$ are neighbors if there exists 
an edge $e \in E_{\mathfrak{g}}$ satisfying $s(e)=v$ and $t(e)=w$.
The valency of a vertex $v$ is the cardinality of the set of edges having $v$ as a source.
The pairs $\lbrace e, r(e) \rbrace$ of mutually reverse edges are called geometric edges of $\mathfrak{g}$.
A finite line in $\mathfrak{g}$, usually denoted by $\mathfrak{p}$, is a 
subgraph whose vertex set $V_{\mathfrak{p}}=\lbrace v_i \rbrace_{i=0}^{n}$ 
satisfies that $v_i$ and $v_{i-1}$ are neighbors, for all $0 \leq i \leq n-1$, 
and $v_i \neq v_j$, for $i \neq j$.
A graph $\mathfrak{g}$ is connected if, given any two vertices $v,w \in 
V_{\mathfrak{g}}$, there exists a finite line $\mathfrak{p}$ connecting them. 
A cycle in $\mathfrak{g}$ is a finite line with an additional geometric edge connecting $v_n$ with $v_0$. 
A tree is a connected graph without cycles.

Additionally, we define a ray $\mathfrak{r}$ in $\mathfrak{g}$ 
as a graph with a vertex set $V_{\mathfrak{r}}=\lbrace v_i  | \, \, i \in \mathbb{Z} \rbrace$, all distinct, with $v_i$ and $v_{i-1}$ neighbors.
The initial vertex of $\mathfrak{r}$ is the vertex $v_0$ above.
Let $\mathfrak{r}_1$ and $\mathfrak{r}_2$ be two rays whose vertex sets are 
respectively denoted by 
$V_1=\left\lbrace v_i | \, \,  i\in \mathbb{Z}_{\geq 0}\right\rbrace $ and 
$V_2=\left\lbrace v'_i | \, \,  i\in \mathbb{Z}_{\geq 0}\right\rbrace$. 
We say that $\mathfrak{r}_1$ and $\mathfrak{r}_2$ are equivalent if there exists 
$t, i_0 \in \mathbb{Z}_{\geq 0}$ such that $v_{i}= v_{i+t}'$, for all 
$i \geq i_0$.
In this case we write $\mathfrak{r}_1 \sim \mathfrak{r}_2$.
We define the visual limit, also called end set, 
$\partial_{\infty}(\mathfrak{g})$ of $\mathfrak{g}$ as the set of equivalence
classes of rays $\mathfrak{r}$ in $\mathfrak{g}$.
The class of $\mathfrak{r}$ is called visual limit of $\mathfrak{r}$ and denoted 
by $\partial_{\infty}(\mathfrak{r})$. 
When consider a quotient graph $G\backslash\mathfrak{g}$,
it is not necessarily the case that we can identify the visual limit
$\partial_\infty(G\backslash\mathfrak{g})$ with the corresponding
quotient set $G\backslash\partial_\infty(\mathfrak{g})$.
For instance, the quotient of
an infinite graph might be finite, and therefore the visual limit might be empty. However, it is possible to show that every ray in $G\backslash\mathfrak{g}$ can be lifted to a ray in $\mathfrak{g}$,
whence we can identify $\partial_\infty(G\backslash\mathfrak{g})$
with a quotient of the form $G\backslash L$ for $L\subseteq
\partial_\infty(\mathfrak{g})$.

A maximal path in a graph $\mathfrak{g}$ is a doubly infinite line, 
i.e. the union of two rays with the same initial vertex, that 
intersect only in that vertex.
For each pair of different visual limits $(a,b)$ of a tree $\mathfrak{t}$ 
there exists a unique maximal path $\mathfrak{p}[a,b]$ connecting them, 
i.e., such that $\partial_{\infty}(\mathfrak{p}[a,b])=\lbrace a,b \rbrace$.
If $(a,b,c)$ is a triplet of  different visual limits 
of $\mathfrak{t}$,
the intersection of the paths $\mathfrak{p}[a,b]$
and $\mathfrak{p}[a,c]$ is a ray whose visual limit is $a$.
Let $w$ be the initial vertex of this ray. Then $w$
is the unique vertex in the graph-theoretical intersection 
$\mathfrak{p}[a,b] \cap  \mathfrak{p}[b,c] \cap \mathfrak{p}[a,c]$,
and it is called the incenter $\mathrm{In}(a,b,c)$ of the triplet $(a,b,c)$.

\section{The Bruhat-Tits tree of $\mathrm{SL}_2$}\label{subsection BT}

Let $F$ be field with a discrete valuation map $\nu: F \to \mathbb{Z} \cup \lbrace \infty \rbrace$. We mostly assume $F=K_\infty$ as in \S~\ref{section intro}.
An example of tree is the Bruhat-Tits tree (building) $\mathfrak{t}=\mathfrak{t}(F)$ associated to the reductive group $\mathrm{SL}_2$ and the field $F$ 
(c.f. \cite[Chap. II, \S 1]{SerreTrees} or \cite{BT1}).
We can interpret the Bruhat-Tits tree $\mathfrak{t}$ from the
topological structure of $F$.
Indeed, the vertex set of $\mathfrak{t}$ corresponds to the set of closed balls 
in $F$. In 
order to keep the notations simpler, we often make the notations for vertices
match the notations of the corresponding ball. For example, the ball of center $a$
and radius $|\pi^r|$ is denoted $B_a\left[|\pi|^r\right]$, or simply $B_{a}^{[r]}$.
Hence, the corresponding vertex is denoted $w_{a}^{[r]}$.

This topological interpretation of $\mathfrak{t}$ is very useful in order to have a concrete representation of its visual limit. 
Indeed, let $\pi=\pi_F$ be a uniformizing parameter of $F$.
The set of neighbors of $w_{a}^{[r]}$ is the set 
$\left \lbrace w_{a+ s \pi^{r}}^{[r]} \Big| s \in S \right \rbrace \cup \left 
\lbrace w_a^{[r-1]} \right \rbrace$, where $S$ is a representative system 
of the residue field of $F$.
The upwards edge of $w_{a}^{[r]}$ is the edge $e$ satisfying 
$s(e)=w_{a}^{[r]}$ and $t(e)=w_{a}^{[r-1]}$. Any edge of this form is called 
an upwards edge. 
The ascending ray starting from $w_a^{[r]}$ is the ray whose vertex set is
$V=\left\lbrace w_{a}^{[r]},w_{a}^{[r-1]},w_{a}^{[r-2]},\dots  \right\rbrace$.
Note that the edge from any of these vertices to the next is an upwards edge.
All ascending rays are equivalent and their visual limit is identified
with the infinite point $\infty\in\mathbb{P}^1(F)$. 
The reverse of an upwards edge is called a downwards edge.
We define a descending ray starting from $w_a^{[r]}$ as a ray whose vertex set 
is $V=\left\lbrace w_{a}^{[r]},w_{a}^{[r+1]},w_{a}^{[r+2]},\dots  \right\rbrace$.
Note that this ray depends on the choice of the element $a$ as the center
of the ball $B_a^{[r]}$, and we identify $a$ with the corresponding visual limit.
The edge from $w_a^{[r]}$ that points towards $a$ is the one whose target is
$w_a^{[r+1]}$.
In the descending ray above, the edge from one vertex to the next is, by definition, a downwards edge pointing towards $a$.
 It is apparent that any ray $\mathfrak{r}$ either is
an ascending ray or contains a descending sub-ray $\mathfrak{r}'$, 
and in the latter case $\mathfrak{r}$ and $\mathfrak{r}'$ are
equivalent.
Next result is straightforwards if one checks that the given vertex
set define, in each case, a maximal path whose set of 
visual limits is precisely $\lbrace a,b \rbrace$:

\begin{lem}\label{lemma pab}
The vertex set of $\mathfrak{p}[a,b]$ is 
\begin{itemize}
\item[(1)] $V=\left \lbrace w_{a}^{[r]} \Big| r \geq \nu(a-b) \right \rbrace \cup \left \lbrace w_{b}^{[r]} \Big| r \geq \nu(a-b) \right \rbrace$, when $a,b \in F$, and
\item[(2)] $V=\left \lbrace w_{a}^{[r]} \Big| r \in \mathbb{Z} \right \rbrace$, when $b=\infty$.\qed
\end{itemize}
\end{lem}

In the sequel, we denote by $B_{a,b}$ the smallest closed ball in $F$ simultaneously contained $a,b \in F$ with $a \neq b$.
More explicitly, we have $B_{a,b}=B_{a}^{[\nu(a-b)]}=B_{b}^{[\nu(a-b)]}$.
In particular $B_{a, b}=B_{b, a}$. We write $w_{a,b}$
for the vertex corresponding to $B_{a,b}$.
By inspection, using 
Lemma~\ref{lemma pab} to check that $w_{a,b}$
is in all three paths in either case, 
we concludes next result:

\begin{lem}\label{lemma incenter}
Let $(a,b,c)$ be a triplet of different points 
in $\mathbb{P}^1(F)$.
The incenter $\mathrm{In}(a,b,c)$ of a triplet $(a,b,c)$ 
equals to $w_{a,b}$ if either $c=\infty$, or if
$a,b,c\in F$ and $\nu(a-b) \geq \nu(b-c), \nu(a-c)$.\qed
\end{lem}

For details on the action of $\mathrm{GL}_2(F)$ on vertices,
as much as its extension to visual limits we follow \cite[\S 4]{ArenasArenasContreras}.
The group $\mathrm{GL}_2(F)$ acts on the tree $\mathfrak{t}$ via simplicial maps, with scalar matrices acting trivially. 
The action of a matrix $\mathbf{g}$ on vertices is denoted by
$v\mapsto \mathbf{g} * v$.
The action on visual limits described via Moebius transformations on $\mathbb{P}^1(F)$. More precisely, the matrix $\mathbf{g}=\sbmattrix abcd$
acts on visual limit via $z\mapsto \gamma(z)=\frac{az+b}{cz+d}$. 
We emphasize this relation by writing $\mathbf{g}=\mathbf{g}_\gamma$.
This is natural since the group $\mathcal{M}(F)$ of Moebius transformations
is isomorphic to the quotient 
$\mathrm{PGL}_2(F) \cong \mathrm{GL}_2(F) /F^{*}$,
so $\mathbf{g}_\gamma$ is a lifting of $\gamma$. The lifting is not
unique, but any other lifting has the form $\lambda\mathbf{g}_\gamma$,
for $\lambda\in F^*$, so the abuse of notation should cause no confusion.
The main reason to use the functional notation for Moebius transformations
is that we use the derivative $z\mapsto \gamma'(z)$  as a computational
tool in later sections. With these notations, 
next lemma is immediate from the definitions:

\begin{lem}\label{lemma inc MT}
For each $\gamma \in \mathcal{M}(F)$ and each triplet $(a,b,c)$ of different points in $\mathbb{P}^1(F)$, we have that 
$\mathbf{g}_\gamma * \mathrm{In}(a,b,c)=\mathrm{In}\Big(\gamma(a),\gamma(b),\gamma(c)\Big) $, for any lifting $\mathbf{g}_\gamma$ of $\gamma$.\qed
\end{lem}

\section{Arithmetic Quotients}\label{subsection quotients}

As in \S~\ref{section intro}, let $X$ be a smooth irreducible projective curve defined over a finite field $\mathbb{F}$ and let $K$ be its function field.
The closed point $P_\infty$ of $\mathcal{C}$ induces a discrete valuation map $\nu=\nu_{\infty}: K \to \mathbb{Z} \cup \lbrace \infty \rbrace$.
Let $F=K_{\infty}$ be the completion of $K$ with respect to $\nu_{\infty}$.
We denote by $\mathcal{O}_{\infty}$ its integer ring, i.e. $\mathcal{O}_{\infty}= \lbrace x \in K_{\infty} | \, \, \nu_{\infty}(x) \geq 0 \rbrace$. 
Let $\mathfrak{t}=\mathfrak{t}(K_{\infty})$ be Bruhat-Tits tree defined in \S~\ref{subsection BT}.
The group $\mathrm{GL}_2(K)$ acts on $\mathfrak{t}$ via simplical maps.
There exists a bipartition on the vertex set of $\mathfrak{t}$ that is preserved by each subgroup $H \subseteq \mathrm{GL}_2(K)$ such that $\mathrm{det}(H) \subseteq K_{\infty}^{2*} \mathcal{O}_{\infty}^{*}$.
The group $H$ acts on $\mathfrak{t}$ without inversions, so the quotient graph $H \backslash \mathfrak{t}$ is well defined.

Let $A$ be the ring of functions of $\mathcal{C}$ that are regular outside $\lbrace P_\infty \rbrace$ and set $\Gamma=\mathrm{GL}_2(A)$.
Since $\det(\Gamma)= \mathbb{F}^{*}$, the quotient graph $\Gamma \backslash \mathfrak{t}$ is well-defined.
This graph has been studied by Serre in order to describe the involved group $\Gamma$.
In the sequel, we call it the \textbf{S}-graph (or Serre's graph).
The following definition, which is taken from \cite[Def. 4.1]{CBHecke}, is very useful for describing the action of $\Gamma$ on $\mathfrak{t}$, and therefore also to describe the \textbf{S}-graph:

\begin{defn}\label{def good quotient}
Let $H$ be a subgroup of $\mathrm{GL}_2(K)$.
We say that $H$ is a \emph{group closing enough umbrellas} (GCEU) if there exists a finite family of rays $ \mathfrak{R}_{H}= \lbrace \mathfrak{r}_i \rbrace_{i=1}^{\gamma} \subset \mathfrak{t}$, each with a vertex set $\lbrace v_{i,n} \rbrace_{n>0}^{\infty}$, where $v_{i,n}$ and $v_{i, n+1}$ are neighbors, satisfying each of the following statements:
\begin{itemize}
    \item[(a)] The set of visual limits $\{\partial_\infty(\mathfrak{r})|
    \mathfrak{r}\in \mathfrak{R}_{H}\}$ is a representative 
    system for $H \backslash \mathbb{P}^1(K)$.
    \item[(b)] $H \backslash \mathfrak{t}$ is obtained by attaching all the images $\overline{\mathfrak{r}_i} \subseteq H \backslash \mathfrak{t}$ to a certain finite graph $Y_{H}$.
    In particular, $\overline{\mathfrak{r}_i} \cap \overline{\mathfrak{r}_j} = \emptyset$, 
    for each $i \neq j$.
    \item[(c)] No $\mathfrak{r}_i$ contains a pair of vertices in the same $H$-orbit.
    \item[(d)] For each $i$ and each $n>0$, we have $ \mathrm{Stab}_H(v_{i,n})\subseteq \mathrm{Stab}_H(v_{i,n+1})$.
    \item[(e)] $\mathrm{Stab}_H(v_{i,n})$ acts transitively on the set of  neighboring vertices, in $\mathfrak{t}$, of $v_{i,n}$, other than $v_{i,n+1}$.
\end{itemize}
A consequence of the preceding definition is that, for any GCEU $H$,
the visual limit $\partial_\infty(H\backslash\mathfrak{t})$ can be identified
with the finite set $H \backslash \mathbb{P}^1(K)$, so the lift of any
ray in the quotient is a ray with a rational visual limit, and any rational
visual limit arises in this fashion.
Note that the notion of ``closing umbrellas'' corresponds to conditions (d) and (e), while (a), (b) and (c) convey the idea of ``closing \textit{enough} umbrellas'', so as to have a ``good'' quotient graph.
\end{defn}

In order to describe a suitable set of rays in the $\mathbf{S}$-graph, we use standard algebraic-geometry notations.
A divisor $D$ on $X$ is a formal sum of the form $D=\sum_{i=1}^r n_i P_i$, where $n_i \in \mathbb{Z}$ and $P_i$ is a closed point of $X$.
The set $\lbrace P_i \rbrace_{i=1}^r$ presented above is called the support of $D$.
The degree of $D$ is, by definition, the integer $\deg(D):=\sum_{i=1}^r n_i \deg(P_i)$, where $\deg(P_i)$ is the degree of the extension $\mathbb{F}[P_i]/ \mathbb{F}$ of the residue field $\mathbb{F}[P_i]$ at $P_i$.
The principal divisor $\mathrm{div}(f)$, for a function $f \in K^*$, is by definition the sum $\mathrm{div}(f):= \sum_{Q \in X} \nu_Q(f) Q$, where $\nu_Q$ is the valuation map defined from the closed point $Q \in X$.
We say that two divisors $D,B$ are linearly equivalent, and we write $D \sim B$, exactly when $D-B=\mathrm{div}(f)$, for some $f \in K^*$.
Since the degree of a principal divisor is zero, we have that $\deg(D)=\deg(B)$ whenever $D \sim B$.
We also write $D \succcurlyeq B$ when $D - B = \sum_{i=1}^r n_i P_i$ with $n_i \geq 0$.
The line bundle $\mathfrak{L}^D$ on $X$ defined from the divisor $D$ is the bundle given, at every open subset $U \subset X$, by
\begin{equation}\label{eq Line bundles}
\mathfrak{L}^B(U)=\lbrace f \in K^* \, |  \, \,  ( \mathrm{div}(f) + B )  \lvert_U \succcurlyeq 0
\rbrace \cup \lbrace 0 \rbrace.  
\end{equation}
These are usually called invertible bundles, and have been extensively
studied in existing literature. They can be seen either as the projective equivalent of ideals, or as a multiplicative
version of divisors, as illustrated by the following properties:
\begin{itemize}
\item  Two divisors $B$ and $D$ are linearly equivalent if and only if $\mathfrak{L}^B$ and $\mathfrak{L}^D$
are isomorphic as line bundles,
\item for any pair $(B,D)$ of divisors, we have $\mathfrak{L}^B\mathfrak{L}^D= \mathfrak{L}^{B+D}$, 
\item we have $\mathfrak{L}^B(U)\subseteq\mathfrak{L}^D(U)$, for all open sets $U$, precisely when
$D \succcurlyeq B$ and
\item $\mathfrak{L}^{\mathrm{div}(g)}=g^{-1}\mathcal{O}_X$.
\end{itemize}
Note that the product $\mathfrak{L}^B\mathfrak{L}^D$ in (2) is defined locally, on open sets $U$, by the relation 
$(\mathfrak{L}^B\mathfrak{L}^D)(U)=\mathfrak{L}^B(U)\mathfrak{L}^D(U)$. With this definition, the bundle
 $\mathfrak{L}^B\mathfrak{L}^D$ is isomorphic to the
tensor product $\mathfrak{L}^B\otimes_{\mathcal{O}_X}\mathfrak{L}^D$.
We define the Picard group $\mathrm{Pic}(X)$ of $X$ as the set of linear equivalence classes of divisors on $X$, with the sum as the composition law. It can be equivalently defined as the set of isomorphism classes of line bundles, with the tensor product as the composition law.
The Picard group of $A$ is, by definition, the quotient $\mathrm{Pic}(X) / \langle \overline{P_\infty} \rangle$ of $\mathrm{Pic}(X)$ by the group generated by the image of $D=P_\infty$ in $\mathrm{Pic}(X)$.
This is isomorphic to the group of linear equivalence classes of divisors on $U_{\infty}=X \smallsetminus \lbrace P_\infty \rbrace$, which are obtained by ignoring the infinite coordinate.
Since $X$ has dimension one, $\mathrm{Pic}(A)$ also coincides with the class group $\mathrm{Cl}(A)$ of the Dedekind domain $A$.
Thus, $\mathrm{Pic}(A)$ is finite, since $\mathbb{F}$ is finite.
Next result follows from \cite[Chapter II, \S 2.1- \S 2.3]{SerreTrees}:

\begin{theoremname}\label{teo serre quot}
The group $\Gamma=\mathrm{GL}_2(A)$ is a GCEU.
In particular, the $\mathbf{S}$-graph $\Gamma \backslash \mathfrak{t}$ is the union of a finite graph $Y$ and a family of rays called cusps.
The set of such cusps can be indexed by the elements
of the Picard group $\mathrm{Pic}(A) = \mathrm{Pic}(X) / \langle \overline{P_\infty} \rangle$, provided redundant cusps are avoided. \qed
\end{theoremname}

Since $\mathfrak{t}$ has a $\Gamma$-invariant bipartition, we have that the $\mathbf{S}$-graph is also bipartite.
Let $\tilde{\mathfrak{r}}_{\infty}$ be the ray with vertices
$\tilde{v}_j=w_0^{[-j]}$ with $j \in \mathbb{Z}_{\geq 0}$.
Next result follows from \cite[Ex. 6, \S 2.1, Ch. II]{SerreTrees}:

\begin{lem}\label{lemma stab in gamma}
We have $\mathrm{Stab}_{\Gamma}(\tilde{v}_0)=\mathrm{GL}_2(\mathbb{F})$, while, for $i>0$, we have:
$$\mathrm{Stab}_{\Gamma}(\tilde{v}_i)= \left \lbrace  \sbmattrix{\alpha}{c}{0}{\beta} \,  \bigg|
\begin{array}{lr}
      \alpha, \beta \in \mathbb{F}^{*}, \\
        c \in A[i] 
    \end{array}
\right \rbrace, $$
where $A[i]=\lbrace x \in A \,  | \, \,  \nu_{\infty}(x) \geq -i \rbrace$.
\end{lem}

In order to describe the image $\mathfrak{r}_{\infty}$ of $\tilde{\mathfrak{r}}_{\infty}$ in the $\mathbf{S}$-graph $\Gamma \backslash \mathfrak{t}$, we introduce a sheaf theoretical interpretation of the 
Bruhat-Tits tree. It is only used in this section.
See \cite[\S 2]{ArenasBravo} or \cite[\S 5]{CBHecke} for more details.
Recall that an $\mathcal{O}_X$-module is
a sheaf of abelian groups such that $\Lambda(U)$ is a 
$\mathcal{O}_X(U)$-module, for all open set $U \subseteq X$,
in a way that scalar multiplication is coherent with
restriction maps.
An $\mathcal{O}_X$-lattice $\Lambda$ on a $K$-vector space $V$ is an
$\mathcal{O}_X$-module  whose generic fiber is $V$.
An $\mathcal{O}_X$-order $\mathfrak{D}$ on an algebra 
$\mathbf{A}$ is an $\mathcal{O}_X$-lattice in $\mathbf{A}$
that is also a sheaf of rings.
An order $\mathfrak{D}$ is maximal if $\mathfrak{D}(U)$ is a maximal order, set-theoretically, for every open 
set $U$. The maximal orders on $\mathbb{M}_n(K)$ are precisely the
sheaves of endomorphisms of the lattices on $K^n$. More precisely,
every maximal order has the form $\mathfrak{D}=
\mathfrak{End}_{\mathcal{O}_X}(\Lambda)$, where
$$\big(\mathfrak{End}_{\mathcal{O}_X}(\Lambda)\big)(U)
= \mathrm{End}_{\mathcal{O}_X(U)}\big(\Lambda(U)\big),$$
for every open set $U$. If we set 
$U_{\infty}=X\smallsetminus\{P_\infty\}$, as in \S~\ref{section intro}, 
then the vertices of the Bruhat-Tits tree at $P_\infty$ are in 
correspondence with the set $\Xi$ of maximal orders $\mathfrak{D}$ 
satisfying $\mathfrak{D}(U_{\infty})=\mathbb{M}_2(A)$. 
The conjugacy classes of maximal 
orders with a representative in $\Xi$ are, therefore, in correspondence
with the orbits of vertices in the tree under the normalizer 
$\mathcal{N}\subseteq\mathrm{GL}_2(K)$ of $\mathbb{M}_2(A)$.
The quotient graph $\mathcal{N}\backslash\mathfrak{t}$
is called the classifying graph, or $\mathbf{C}$-graph, for this reason. In general, the group $\mathcal{N}$ fails
to preserve the bipartition, so this quotient
graph must be defined in terms of the baricentric
subdivision. We skip the details of this construction 
here, but we note that the $\mathbf{C}$-graph might include
loops or half-edges. The half edges can be interpreted
as ``folded edges'' and are associated to the existence 
of inversions.
The group $\mathcal{N}$ contains $\Gamma$ as a normal subgroup of finite index.
This implies the existence of a finite
ramified covering from the $\mathbf{S}$-graph to the 
$\mathbf{C}$-graph.
The maximal $X$-order defined by the divisor $B$ on $X$, is the sheaf:
$$\mathfrak{D}_B=\sbmattrix{\mathcal{O}_X}{\mathfrak{L}^B}{\mathfrak{L}^{-B}}{\mathcal{O}_X}
=\mathfrak{End}_{\mathcal{O}_X}\textnormal{\scriptsize$\left(
\begin{array}{c}  \mathfrak{L}^B\\ \mathcal{O}_X\end{array}\right)$\normalsize},$$
This is a locally free $\mathcal{O}_X$-module whose generic fiber is $\mathbb{M}_2(K)$.
It corresponds to a vertex in the tree precisely when $B$ is
supported at $P_\infty$.
It follows from \cite[Prop. 5.3]{ArenasBravo} that the 
corresponding vertices belong to the same $\mathcal{N}$-orbit 
exactly when the corresponding orders are conjugates.

\begin{lem}\label{lemma standard ray}
The graph $\mathfrak{r}_{\infty}$ is a ray.
\end{lem}

\begin{proof}
Since the action of $\Gamma$ on $\mathfrak{t}$ is simplicial, we just have to prove that two different vertices $\tilde{v}_i$ and $\tilde{v}_j$ do not belong to the same $\Gamma$-orbit.  
Firstly, if $\tilde{v}_0= \mathbf{g} \cdot \tilde{v}_i$, with $\mathbf{g} \in \Gamma$ and $i>0$, we have that $\mathrm{Stab}_{\Gamma}(\tilde{v}_0) = \mathbf{g} \mathrm{Stab}_{\Gamma}(\tilde{v}_i) \mathbf{g}^{-1}$.
In particular, these group stabilizers are isomorphic, which is absurd according to Lemma~\ref{lemma stab in gamma} since $\mathrm{Stab}_{\Gamma}(\tilde{v}_i)$ is not simple.

Now, let $i,j>0$ be two different integers and let $B = i \cdot P_\infty$ 
and $D= j \cdot P_\infty$ be two divisors on $X$.
Then, as noted above, the vertices  $\tilde{v}_i$ and 
$\tilde{v}_j$ belongs to the same $\mathcal{N}$-orbit exactly when $\mathfrak{D}_B$ and $\mathfrak{D}_D$ are conjugates.
This holds in particular if they are in the same $\Gamma$-orbit.
Since the degree of a principal divisor $\mathrm{div}(f)$ is zero, we conclude from Eq.~\eqref{eq Line bundles} that $\mathfrak{L}^{-B}(X)=\mathfrak{L}^{-D}(X)=\lbrace 0 \rbrace$.
Then, it follows from \cite[Prop. 4.1]{ArenasQuotient}, or alternatively from \cite[Lemma 6.3]{ArenasBravo}, that $B$ is linearly equivalent to $D$ or $-D$.
This is absurd, since the degree of both $B-D=(i-j) \cdot P_\infty$ and $B+D= (i+j) \cdot P_\infty$ is non-zero.
The result follows.
\end{proof}

In the sequel, we refer to $\mathfrak{r}_{\infty}$ as the standard ray of the \textbf{S}-graph.
The image of $\tilde{v}_j$ in $\mathfrak{r}_{\infty}$ is denoted by $\V_j$, for all $j\in \mathbb{Z}_{\geq 0}$.
The sub-ray $\mathfrak{r}_{\infty}^{s} \subseteq \mathfrak{r}_{\infty}$ containing precisely those of vertices $\V_j$ with $j>N$ is called the strict standard ray of the \textbf{S}-graph. It plays a significant role in the sequel.

\begin{lem}\label{lemma strict standard ray}
For each $j >N$, the subgroup $\mathrm{Stab}_{\Gamma}(\tilde{v}_j)$ acts transitively on the set $\mathfrak{v}_j$ of neighboring vertices
$\tilde{v}\neq\tilde{v}_{j+1}$ of $\tilde{v}_j$,
in $\mathfrak{t}$.   
\end{lem}

\begin{proof}
For each $i \in \mathbb{Z}_{\geq 0}$, let us introduce the groups
$$\Delta_i:= \left\lbrace \sbmattrix {1}{z}{0}{1} \bigg| z \in \pi^{-i}\mathcal{O}_{P_{\infty}} \right \rbrace \quad\text{ and }\quad U_i:= \left \lbrace  \sbmattrix{1}{c}{0}{1} \bigg| \, c \in A[i]
\right \rbrace.$$
Since $A[i]= A \cap \pi^{-i} \mathcal{O}_{P_{\infty}}$, we have 
that $U_i \subset \Delta_i$. Note that $\Delta_i$ fixes $\tilde{v}_i$,
but moves every neighbor apart from $\tilde{v}_{j+1}$. Likewise,
$\Delta_{i-1}$ is the $\Delta_i$-stabilizer of every vertex in $\mathfrak{v}_i$. 
In particular, $\Delta_{j}/\Delta_{j-1}$ is a group
acting freely on a set of the same size, whence this action 
is transitive.
Note that $U_{j}$ projects surjectively $\Delta_{j}/\Delta_{j-1}$, since $A[i]/ A[i-1] = \pi^{-i}\mathcal{O}_{P_{\infty}} / \pi^{-i+1}\mathcal{O}_{P_{\infty}} \cong \mathbb{F}[P_{\infty}]$ according to Riemann-Roch Theorem (See \cite[Lemma 1.2]{Mason}).
Since $U_i \subseteq\mathrm{Stab}_{\Gamma}(\tilde{v}_i)$, the result follows.
\end{proof}

\begin{ex}\label{ex S and C graph in degree two}
It follows from \cite[Ch. II, \S 2.4.2]{SerreTrees} that when $X=\mathbb{P}^1_{\mathbb{F}}$, and $P_\infty$ is a point of
degree $2$, the $\mathbf{S}$-graph is a maximal path, as shown in
Figure~\ref{Figure degree two}(A). In this case the quotient group
$\mathcal{N}/\Gamma$ has order two, and act by sending the vertex
$\V_i$ to $\V_{i-1}^*$. Each vertex $\mathbf{w}_i$ is the image of 
both $\mathbf{v}_i$ and $\mathbf{v}_{i-1}^{*}$. For details 
of this computation see \cite[Example 5.8]{ArenasQuotient}. 
\end{ex}

\begin{figure}
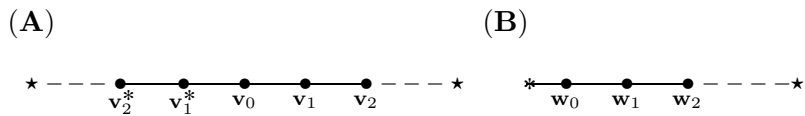

\[
\xygraph{!{<0cm,0cm>;<0.8cm,0cm>:<0cm,0.8cm>::}
!{(-1.5,1) }*+{\mathbf{(A)}}="B"
!{(-1.5,0) }*+{\star}="vl1"
!{(5.5,0) }*+{\star}="vl2"
!{(-.2,0) }*+{}="v0"
!{(-0.04,-0) }*+{\bullet}="v1"
!{(1,0) }*+{\bullet}="v2"
!{(2,0) }*+{\bullet}="v3"
!{(3,0) }*+{\bullet}="v4"
!{(4,0) }*+{\bullet}="v5"
!{(4.2,0) }*+{}="v5f"
!{(0,-0.3) }*+{{}_{\V_2^*}}="v1i"
!{(1,-0.3) }*+{{}_{\V_1^*}}="v2i"
!{(2,-0.3) }*+{{}_{\V_0}}="v3i"
!{(3,-0.3) }*+{{}_{\V_1}}="v4i"
!{(4,-0.3) }*+{{}_{\V_2}}="v5i"
!{(5.2,0) }*+{}="v7"
"vl1"-@{--}"v1" "v0"-"v5f" "vl2"-@{--}"v5" 
}
\xygraph{!{<0cm,0cm>;<0.8cm,0cm>:<0cm,0.8cm>::}
!{(1,1) }*+{\mathbf{(B)}}="B"
!{(1.4,0) }*+{*}="v1/2"
!{(2,0) }*+{\bullet}="v3"
!{(3,0) }*+{\bullet}="v4"
!{(4,0) }*+{\bullet}="v5"
!{(5.8,0) }*+{\star}="v7"
!{(1.6,0) }*+{}="v3f"
!{(2.2,0) }*+{}="v3ff"
!{(4.2,0) }*+{}="v5f"
!{(1.2,0) }*+{}="v1/2f"
!{(2,-0.3) }*+{{}_{\mathbf{w}_0}}="v3i"
!{(3,-0.3) }*+{{}_{\mathbf{w}_1}}="v4i"
!{(4,-0.3) }*+{{}_{\mathbf{w}_2}}="v5i"
!{(5.8,0) }*+{}="v7f"
"v3f"-@{-}"v5f" "v5"-@{--}"v7f" "v3ff"-@{-}"v1/2f"
}
\]
\caption{
Figure~(A) shows the $\mathbf{S}$-graph when 
$\mathbb{P}^1_{\mathbb{F}}$ and $\deg(P_{\infty})=2$.
Figure~(B) is the corresponding classifying graph, which has a semi-edge joined to $\mathbf{w}_0$.
}\label{Figure degree two}
\end{figure}

\section{The walk of a local integer}\label{section walk}

Every descending ray $\mathfrak{r}$ starting from $\tilde{v}_0=w_0^{[0]}$ has a 
visual limit $\alpha(\mathfrak{r})\in\mathcal{O}_\infty$. Conversely,
every element in the ring $\mathcal{O}_\infty$ can be seen as the visual limit of
a unique such ray $\mathfrak{r}_\alpha$. The vertices in the ray
$\mathfrak{r}_\alpha$ are $\tilde{v}_0(\alpha),
\tilde{v}_1(\alpha),\tilde{v}_2(\alpha),\dots$, 
where $\tilde{v}_i(\alpha)=w_{\alpha}^{[i]}$. In other words, 
$\tilde{v}_{i+1}(\alpha)$ corresponds to a proper maximal sub-ball of $B_\alpha^{[i]}$, and therefore a
neighbor of $\tilde{v}_i(\alpha)$ in the Bruhat-Tits tree.
We denote by
$v_i(\alpha)$ the image of $\tilde{v}_i(\alpha)$ under the 
canonical projection $\phi:\mathfrak{t}\rightarrow \Gamma\backslash\mathfrak{t}$.
In particular, the vertex $v_{i+1}(\alpha)$ is always a neighbor of $v_i(\alpha)$.
We refer to the sequence $v_0(\alpha)v_1(\alpha)v_2(\alpha)\dots$
as the walk of $\alpha$.

\begin{defn}\label{nmul}
For any pair $(v,w)$ of vertices in $\Gamma\backslash\mathfrak{t}$,
we let $m_{v,w}$ be the number of pre-images of $w$ that are neighbors of a given
pre-image $\tilde{v}$ of $v$. Note that this number is non-zero 
precisely when $v$ and $w$ are neighbors. It is also independent on the 
choice of the  pre-image $\tilde{v}$.   
We call it the neighbor multiplicity.
\end{defn}

 The explicit description
of $\Gamma\backslash\mathfrak{t}$ in Theorem~S~\ref{teo serre quot} has the
following consequence that we use throughout:

\begin{lem}\label{L61}
   The graph $\Gamma\backslash\mathfrak{t}$ is the union of a finite graph,
   which we call the center in the sequel, and a finite number of 
   non-equivalent rays
   $\mathfrak{c}_0,\dots,\mathfrak{c}_m$ called cusps. Each cusp $\mathfrak{c}$
   has a vertex set $v_0,v_1,v_2,\dots$, satisfying the following conditions:
   \begin{itemize}
\item[(a)] The only neighbors of $v_i$, for $i\geq1$, are 
       $v_{i-1}$ and $v_{i+1}$.
\item[(b)] For all $i\geq1$, we have $m_{v_i,v_{i+1}}=1$ and 
       $m_{v_i,v_{i-1}}=q^{\deg(P_\infty)}$.
   \end{itemize}
   The vertex $v_0$, which is called the initial vertex and belong to 
   the center, is allowed to have additional neighbors. \qed
\end{lem}

\begin{lem}\label{L62}
Two cusps, as defined above, can intersect in the initial vertex at most.
\end{lem}

\begin{proof}

Let $\mathfrak{c}$ and $\mathfrak{c}'$ be two cusps with respective 
vertex sets $\lbrace v_i \rbrace_{i=0}^{\infty}$ and $\lbrace v'_i \rbrace_{i=0}^{\infty}$.
  As the cusps are non-equivalent  rays, condition (a) implies that
  the only way there could be a non-trivial intersection wold be
  $v_t=v'_0, v_{t-1}=v'_1,\dots,v_0=v'_t$ for some $t>0$.
  Since every non-initial vertex in a cusp has exactly two neighbors, 
  this condition precludes us to have additional edge in any vertex
  of the union $\mathfrak{c}\cup\mathfrak{c}'$.
  In this case every vertex in $\Gamma\backslash\mathfrak{t}$ belong 
  to either cusp, whence $\Gamma\backslash\mathfrak{t}=
  \mathfrak{c}\cup\mathfrak{c}'$ is a maximal
  path. Condition (b) implies that the stabilizer of every vertex 
  in $\mathfrak{t}$ has
  an orbit of neighbors of size $1$ and an orbit of neighbors of size
  $q^{\deg(P_\infty)}$. This is certainly false for the vertex $\tilde{v}_0$,
  whose stabilizer is $\mathrm{GL}_2(\mathbb{F})$, which acts on
  $\mathbb{P}^1(\mathbb{F}[P_\infty])$ without a size-one orbit (c.f. \cite[Ch. II, \S 2.1, Ex. 6]{SerreTrees}).
  The result follows.
\end{proof}

\begin{rem}
In Example~\ref{E3}, we compute the size of all the $\mathrm{GL}_2(\mathbb{F})$-orbits in $\mathbb{P}^1(\mathbb{F}[P_\infty])$, when $\deg(P_{\infty})=2$.
The fact that there is no size-one orbit for any value of
$\deg(P_{\infty})$ is quite straightforward.
Indeed, an element in a size-one orbit is an invariant point 
for the whole group.
But $\sigma(z)=z+1$ leaves no finite invariant point, while
$\iota(z)=-1/z$ does not leave $\infty$ invariant.
\end{rem}

We use throughout the convention that the cusps, as defined above,
are maximal, i.e., they cannot be enlarged while satisfying the 
conditions in Lemma~\ref{L61}. In other words, we assume the center
is as small as possible. Next result is an easy consequence of
Lemma~\ref{lemma strict standard ray}:

\begin{lem}\label{lemma strict is cuspidal}
The strict standard ray $\mathfrak{r}_{\infty}^s$ is contained in a cusp.\qed
\end{lem}

Since $\tilde{v}_{i+1}(\alpha)$ is a neighbor of $\tilde{v}_i(\alpha)$
that fails to coincide with $\tilde{v}_{i-1}(\alpha)$, we cannot have
$v_{i+1}(\alpha)=v_{i-1}(\alpha)$ unless 
$m_{v_i(\alpha),v_{i-1}(\alpha)}>1$.  
We note that once the walk of an element $\alpha$ enters
a cusp, and give a step back, in the sense that 
$v_{i+1}(\alpha)=v_{i-1}(\alpha)$ is closer to the initial vertex than
$v_i(\alpha)$, then the walk can no longer reverse its direction
again until it reaches the initial vertex $v_0$ of that cusp. 
This remark is essential in all that follows.

\begin{lem}\label{L63}
   If $\alpha\in\mathcal{O}_\infty$ is irrational, the set of positive 
   integers $i$ for which $v_i(\alpha)$ belongs to the center of the graph
   $\Gamma\backslash\mathfrak{t}$ is infinite.   
\end{lem}

\begin{proof}
  The preceding discussion shows that   $v_i(\alpha)$ belongs to the center of
   $\Gamma\backslash\mathfrak{t}$ for infinitely many values of $i$ unless
   the walk corresponding to $\alpha$ eventually enters a cusp $\mathfrak{c}$
   and never gives a step back. We claim that this can only be the case 
   for a rational visual limit $\alpha$.
   This is actually a consequence of (a) in Def.~\ref{def good quotient}, but we include an independent proof here for the sake of completeness.
   Let $v_0,v_1,v_2,\dots$ be the 
   vertices in the cusp $\mathfrak{c}$ and let 
   $\tilde{v}_0,\tilde{v}_1,\tilde{v}_2,\dots$ be the pre-images in 
   $\mathfrak{t}$ corresponding to a ray whose visual limit is $\alpha$.
   
   It follows from (b) in  Lemma~\ref{L61} that, for any $i>0$, there must exist
   non-trivial matrices in $\Gamma$ that permutes all neighbors of 
   $\tilde{v}_i$ other than $\tilde{v}_{i+1}$, while leaving 
   $\tilde{v}_i$ invariant. In particular, this matrices must also
   leave $\tilde{v}_{i+1}$ invariant, since it is the only neighbor of
   $\tilde{v}_i$ in the corresponding class. This reasoning can also be 
   applied to every subsequent vertex $\tilde{v}_{i+2},\tilde{v}_{i+3},\dots$,
   so that they also leave the visual limit $\alpha$ invariant. As an invariant
   value for a Moebius transformation with coefficients in $K$, the element
   $\alpha$ must, at most, be quadratic over $K$.

   If the extension $K[\alpha]/K$ is an inseparable quadratic extension,
   then it is ramified at all places, since fully inseparable extensions of
   function fields of curves over perfect fields have such property. This
   contradicts the fact that $\alpha\in K_\infty$. On the other hand, if
   $K[\alpha]/K$ is a separable quadratic extension, the second root
   $\alpha'$ of the irreducible polynomial of $\alpha$ over $K$ must
   also belong to $K_\infty$. By increasing the value of $i$ in the 
   above argument, we can assume that one of the matrices leaving
   $\tilde{v}_i$ invariant moves the edge that points towards
   $\alpha'$. Thus we ensure the existence of a matrix in $\Gamma$
   whose action leaves $\alpha$ invariant, but not $\alpha'$. This is
   not possible, as they are conjugate roots of an irreducible polynomial.
   We must therefore conclude that $\alpha\in K$.
\end{proof}

\section{Approximation and quotient graphs}\label{section general approximation}

We keep the notation from section \S~\ref{section graphs}.
Indeed, we set $\Gamma=\mathrm{GL}_2(A)$, and we write
$\mathfrak{r}_{\infty}$ for the image of
$\tilde{\mathfrak{r}}_{\infty}$ 
in the \textbf{S}-graph $\Gamma \backslash \mathfrak{t}$
(c.f. Lemma~\ref{lemma standard ray}).
The strict standard ray $\mathfrak{r}_{\infty}^{s}$ is 
the sub-ray whose vertices are precisely the $\mathbf{v}_j$ with 
$j>N=\left\lceil\frac{2g(X)-1}{\deg(P_\infty)}\right\rceil$.
As in \S~\ref{section intro}, we call an  element 
$\alpha \in K_{\infty}$ well approximable when it is the 
limit of a sequence of good approximations, where
a good approximation of $\alpha$ is a rational element $f/g$,
satisfying  both $f A + g A =A$ and $\left|\alpha-\frac{f}{g}\right|_\infty<\frac{|\pi|_{\infty}^N}{|g|_\infty^2}$.

\begin{lem}\label{l61a}
Let $r,g \in A$. If $|r|_\infty\geq |g|_\infty|\pi|_\infty^{-N-1}$, then we can write $r=gu+r'$, with $u,r'\in A$ and $|r|_\infty> |r'|_\infty$.
\end{lem}

\begin{proof}
    We set $|g|_\infty=|\pi|^{-s}_\infty$, and
   $$A_k=\left\{f\in A\Big| |f|_\infty\leq |\pi|_\infty^{-k}\right\},$$
   for every integer $k\geq N$. By Riemman-Roch's Theorem (See 
   \cite[Lemma 1.2]{Mason}),   
   multiplication by $g$ induces an injective linear map
   $m_g:A_{k+1}/A_k\rightarrow A_{k+s+1}/A_{k+s}$ between vector spaces of the same dimension, and it is therefore surjective. 
   Set $k$ by the relation 
   $|r|_\infty=|\pi|_\infty^{-k-s-1}$. If we choose $u\in A_{k+1}$ satisfying
   $m_g\left(u+A_k\right)=r+A_{k+s}$, the result follows.
\end{proof}

For a ball $B=B_u^{[n]}$, we write $n=\nu_{\infty}(B)$.

\begin{lem}\label{nclumov}
Let $\rho$ be a Moebius transformation. Let $B=B_u^{[n]}$ a ball
whose point-wise image $\rho(B)$ is also a ball. Then
$\nu_{\infty}\Big(\rho(B)\Big)=n+\nu_{\infty}(\rho'(u))$.
\end{lem}

\begin{proof}
Set $\rho(z)=\frac{az+b}{cz+d}$.
Moebius transformations are analytic functions whose Taylor expansion at a point $u \in K_{\infty}$
is quite explicit:
$$\rho(z)=\sum_{n=0}^\infty\rho_n(u)(z-u)^n=\frac ac+\frac{bc-ad}{c(cu+d)}
\sum_{n=0}^\infty\left(\frac{c(u-z)}{cu+d}\right)^n.$$
Here $\rho_0=\rho$, while $\rho_1=\rho'$ is the derivative.
Since $\rho(B)$ is a ball, this series is convergent on $B$. It follows that $\nu_{\infty}\left(\frac{c(z-u)}{cu+d}\right)>0$
for any $z\in B$. This implies that $\rho(z)-\rho(u)=\sum_{n=1}^\infty\rho_n(u)(z-u)^n$, and the $n$-th term in this sum is 
$\rho_n(u)(z-u)^n=\frac{bc-ad}{c(cu+d)}\cdot\left(\frac{c(z-u)}{cu+d}\right)^n$,
where the valuation of the power is strictly increasing. Therefore, the first term is
dominant for all $z$. 
We conclude that
$$\nu_{\infty}\Big(\rho(z)-\rho(u)\Big)=\nu_{\infty}\Big(\rho_1(u)(z-u)\Big)=\nu_{\infty}\left(\rho'(u)\right)+
\nu_{\infty}(z-u).$$
Since $z$ is an arbitrary element in the ball, the result follows. 
\end{proof}

In the sequel, we denote by $\iota$ the Moebius transformation defined by $\iota(z)=-z^{-1}$, for all $z \in \mathbb{P}^1(K_{\infty})$. One lifting is the matrix
$\mathbf{g}_\iota=\sbmattrix 0{-1}10 \in \Gamma$. Recall that $B_{\alpha,\beta}$
is the smallest ball containing both $\alpha$ and $\beta$,
while $w_{\alpha,\beta}$ is the corresponding vertex.
See \S~\ref{subsection BT}.

\begin{lem}\label{l61}
Let $\frac fg$, with $f,g\in A$, be a good approximation of
$\alpha$. Then the image of $w_{\frac fg,\alpha}$ in the
quotient graph belongs to the strict
standard ray $\mathfrak{r}_{\infty}^{s}$. More precisely,
there exists a Moebius transformation 
$\eta(z)=\frac{sz-f}{rz-g}$,
with $r,s\in A$ satisfying $fr-gs=1$, for which 
$B_{\alpha,\frac fg}=\eta\left(B_{-\beta^{-1},0}\right)$ with 
$\nu_{\infty}(\beta)>-N$.
\end{lem}

\begin{proof}
Assume $\frac fg$ is a good approximation of $\alpha$.
Since, by definition, $fA$ and $gA$ are co-maximal ideals, there exists elements $r$ and $s$ in $A$ satisfying
$fr-gs = 1$. Hence, the Moebius transformation 
$\mu(z)=\frac{fz+s}{gz+r}$, which has coefficients in $A$,
satisfies both $\mu(\infty)=\frac fg$ and 
$\mu(0)=\frac sr$.
By Lemma \ref{l61a} we can
assume that $|r|_\infty\leq |g|_\infty|\pi|_\infty^{-N}$, or we define $r'=r-gu$
as in the lemma, so we get $fr'-g(s-uf)=1$, 
which can be used to
replace $r$ by the smaller element $r'$.
Set $\beta=\mu^{-1}(\alpha)$.
Note that  $\mu$ maps the triplet 
$(\infty,\beta,0)$ to 
$\left(\frac fg,\alpha,\frac sr\right)$. Then, by Lemma \ref{lemma inc MT},
it must take the incenter $w_{0,\beta}$
of the first triplet to the incenter
of the second. 
Given the inequality satisfied by $|r|_\infty$
and the fact that $\frac fg$ is a good approximation of $\alpha$, we can write
$$\left|\frac sr-\frac fg\right|_\infty=\frac1{|gr|_\infty}\geq
\frac{|\pi|_\infty^{N}}{|g|_\infty^2}>\left|\alpha-\frac fg\right|_\infty.$$
Then, it follows from Lemma~\ref{lemma incenter} that the incenter of the second triplet is $w_{\frac fg,\alpha}$.
Furthermore, the same inequality shows that the transformation
$\eta=\mu\circ\iota$
maps the ball $B_{0,\beta^{-1}}$ onto the ball $B_{\frac fg,\alpha}$. 
It just remains to prove that the image of $w_{0,\beta^{-1}}$ 
belongs to the strict standard ray $\mathfrak{r}_{\infty}^{s}$.
Indeed, an application of the Lemma~\ref{nclumov} shows that
$$\nu_{\infty}\left(\frac fg-\alpha\right)=-\nu_{\infty}(\beta)+\nu_{\infty}\left(\eta'(0)\right)=
-\nu_{\infty}(\beta)+\nu_{\infty}\left(\frac1{g^2}\right).$$
We conclude that $-\nu_{\infty}(\beta)>N$, whence the result follows.    
\end{proof}

Recall that $v_0(\alpha) v_1(\alpha) v_2(\alpha) \cdots$ denotes the walk of $\alpha$, see \S~\ref{section walk}.
Let $m\geq0$ and $M\geq N$ be integers.
We write $v_{m+M+1}(\alpha) \stackrel\nearrow= \V_{M+1}$, when we have both $v_{m+M+1}(\alpha)= \V_{M+1}$ and
$v_{m+M}(\alpha) = \V_{M}$.
By the discussion following Lemma~\ref{lemma strict is cuspidal}, this implies
$v_{m+t}(\alpha) = \V_{t}$ for $N\leq t\leq M+1$.

\begin{lem}\label{PWA2}
Let $M\geq N$.
An irrational element $\alpha$ has an $M$-approximation $\frac fg$
satisfying $\nu_{\infty}(g)=-n$ if and only if $v_{2n+M+1}(\alpha)  \stackrel\nearrow= \V_{M+1}$.
\end{lem}

\begin{proof}
Assume that the latter condition is satisfied for some $n\geq0$.
If we also have $v_{2n+M'+1}(\alpha)=\V_{M'+1}$ for all $M'>M$,
 there is a sub-ray
of $\mathfrak{r}_\alpha$ mapping isomorphically onto the strict 
standard ray, whence $\alpha$ is rational. Therefore, we can 
assume that $M$ is chosen maximal satisfying the hypotheses.
In particular, $v_{2n+M+2}(\alpha)=\V_M$. 
The hypotheses also
gives $v_{2n+M}(\alpha)=\V_M$.
Then, there is a Moebius transformation $\delta$, with
$\mathbf{g}_\delta=\left(\begin{array}{cc}r&-s\\-g&f\end{array}\right)
\in\mathrm{GL}_2(A)=\Gamma$, for which $\mathbf{g}_\delta*\tilde{v}_{2n+M+1}(\alpha)=\tilde{v}_{M+1}
=w_{0}^{[-M-1]}$.
Note that both neighbors,
$\mathbf{g}_\delta*\tilde{v}_{2n+M}(\alpha)$ and 
$\mathbf{g}_\delta*\tilde{v}_{2n+M+2}(\alpha)$, 
correspond to sub-balls of $B_{0}^{[-M+1]}$. Pre-multiplying $\mathbf{g}_\delta$ 
by an element in the stabilizer of 
$\tilde{v}_{M+1}$ if needed, we can assume 
that $\mathbf{g}_\delta*\tilde{v}_{2n+M}(\alpha)=\tilde{v}_M
=w_{0}^{[-M]}$. In particular,
$w=\mathbf{g}_\delta*\tilde{v}_{2n+M+2}(\alpha)$ lies outside of the ray joining
$\tilde{v}_0=w_{0}^{[0]}$ to $\infty$. Let
$\beta=\delta(\alpha)$, so $\alpha=\delta^{-1}(\beta)$, while we have 
$\frac fg=\delta^{-1}(\infty)$ and $\frac sr=\delta^{-1}(0)$.  See
 Figure~\ref{F1}(A).
\begin{figure}
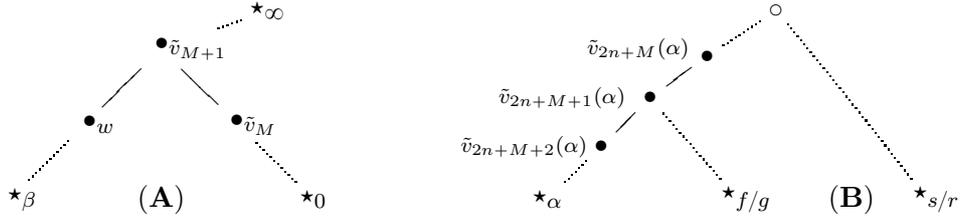

\[
\xygraph{!{<0cm,0cm>;<1cm,0cm>:<0cm,1cm>::}
!{(1,-1) }*+{\mathbf{(A)}}
!{(0,0) }*+{\hphantom{X}\bullet_{w}}="a"
!{(1,1) }*+{\hphantom{XX.}\bullet_{\tilde{v}_{M+1}}}="b"
!{(2,0) }*+{\hphantom{X.}\bullet_{\tilde{v}_{M}}}="c"
!{(3,-1)}*+{\star_{0}}="d"
!{(-1,-1)}*+{\hphantom{X}\star_{\beta}}="e"
!{(2.4,1.5)}*+{\star_{\infty}}="g"
"a"-"b" "c"-"b" "c"-@{.}"d" "a"-@{.}"e" "b"-@{.}"g"}
\qquad\qquad
\xygraph{!{<0cm,0cm>;<1cm,0cm>:<0cm,1cm>::}
!{(3,-1) }*+{\mathbf{(B)}}
!{(-.3,-.3) }*+{\bullet}="a"
!{(-1.3,-.3) }*+{{}_{\tilde{v}_{2n+M+2}(\alpha)}}
!{(.35,.35) }*+{\bullet}="b"
!{(-.8,.35) }*+{{}_{\tilde{v}_{2n+M+1}(\alpha)}}
!{(1.1,.9) }*+{\bullet}="c"
!{(.2,1) }*+{{}_{\tilde{v}_{2n+M}(\alpha)}}
!{(2,1.5) }*+{\circ}="f"
!{(4,-1)}*+{\hphantom{x}\star_{s/r}}="d"
!{(1.5,-1)}*+{\hphantom{x}\star_{f/g}}="g"
!{(-1,-1)}*+{\star_{\alpha}}="e"
"a"-"b" "c"-"b" "f"-@{.}"c" "f"-@{.}"d" "a"-@{.}"e" "b"-@{.}"g"}
\]
\caption{Figures used in the proof of Lemma~\ref{PWA2}}\label{F1}
\end{figure}
We conclude from the figure that $\tilde{v}_{M+1}$ is the incenter 
of the triplet $(\beta,\infty,0)$. 
Then, since $\tilde{v}_{M+1}=\mathbf{g}_\delta*\tilde{v}_{2n+M+1}(\alpha)$, 
the vertex 
$\tilde{v}_{2n+M+1}(\alpha)$ must be the incenter of the triplet
$\left(\alpha,\frac fg,\frac rs\right)$. 
Since $\tilde{v}_{M}=d*\tilde{v}_{2n+M}(\alpha)$ is between $\tilde{v}_{M+1}$ and $0$, then $\tilde{v}_{2n+M}(\alpha)$
is between $\tilde{v}_{2n+M+1}(\alpha)$ and $\frac rs$.
 We conclude that $\frac sr$ is farther away from $\alpha$
than $\frac fg$, as shown in Figure~\ref{F1}(B). 
In particular $\tilde{v}_{2n+M+1}(\alpha)=w_{\alpha,\frac fg}$.

Recall that $\iota$ is the Moebius transformation corresponding to
$\mathbf{g}_\iota=\sbmattrix 0{-1}10$.
Now write $u=\mathbf{g}_\iota*\tilde{v}_{M+1}=w_0^{[M+1]}$, 
which is the incenter of the triplet $(-\beta^{-1},\infty,0)$.
Note that $u$ corresponds to the smallest ball containing $0$ and $-\beta^{-1}$, so its valuation is 
$-\nu_{\infty}(\beta)$. On the other hand, 
$\mathbf{g}_\delta^{-1}\mathbf{g}_\iota \in\mathrm{GL}_2(A)$  
maps the upwards edge of $u$ to the upwards edge of $\tilde{v}_{2n+M+1}(\alpha)$.
In other words, the corresponding Moebius Transformation 
$\eta(z)=\delta^{-1}\circ \iota(z)=\frac{sz-f}{rz-g}$
maps the ball $B_{0,-\beta^{-1}}$, 
corresponding to $u$, as a set,
to the ball $B_{\alpha,\frac fg}$ corresponding to 
$\tilde{v}_{2n+M+1}(\alpha)$, 
so that Lemma~\ref{nclumov} can be applied.
We conclude that
$$2n+M+1=\nu_{\infty}\left(\alpha-\frac fg\right)=
\nu_{\infty}\Big(\eta'(0)\Big)-\nu_{\infty}(\beta)=-\nu_{\infty}(g^2)-\nu_{\infty}(\beta).$$
 Since $\nu_{\infty}(\beta)=-M-1$, this shows the identity,
 $\nu_{\infty}(g)=-n$. Since this gives
 $\nu_{\infty}\left(\alpha-\frac fg\right)>-\nu_{\infty}(g^2)+M$,
 the result follows.

Now we prove the converse. Assume we have an $M_1$-approximation 
$\frac fg$ of $\alpha$ satisfying $\nu(g)=-n_1$. Assume also $M_1\geq N$,
so that, in particular, we have a good approximation, 
whence we can apply Lemma~\ref{l61}. The hypotheses 
tells us 
that, if we set $i=\nu_\infty\left(\alpha-\frac fg\right)$, so that
$\tilde{v}_i(\alpha)=w_{\alpha,\frac fg}$, then $v_i(\alpha)=\V_{M'}$
for some integer $M'>N$. Moreover,
the same lemma shows that the ball 
$B_{\alpha,\frac fg}$ is the image
of a ball of the form $B_{-\beta^{-1},0}$ by a Moebius 
transformation  $\eta(z)=\frac{sz-f}{rz-g}$ with $fr-gs=1$. 
This means that $\tilde{v}_{i-1}(\alpha)=\mathbf{g}_\eta*w$
where  $w=w_0^{[-\nu_{\infty}(\beta)-1]}$, while 
$\mathbf{g}_{\eta^{-1}}$ maps
$w_{f/g}^{[i+1]}$ to $w_0^{[-\nu_{\infty}(\beta)+1]}$ and
$w_\alpha^{[i+1]}$ to the vertex corresponding to a
ball not containing $0$.
This shows that $v_{i-1}(\alpha)=v_{i+1}(\alpha)=\V_{M'-1}$.
Now we can set $\delta=\iota\circ\eta^{-1}$, $M=M'-1$ and 
$2n=i-M-1$, and repeat the computations in the first 
part of the proof. So, as before, we get $2n=-2\nu_{\infty}(g)=2n_1$. In particular $n$ is an integer.
Furthermore, we get $\nu_{\infty}\left(\alpha-\frac fg\right)=-\nu_{\infty}(g^2)+M+1$, so the fact that $\frac fg$ is an
$M_1$-approximation of $\alpha$ shows that $M_1<M+1$.
Now an inductive argument shows that $v_{2n+M+1-j}(\alpha)=\V_{M+1-j}$
for $0\leq j\leq M+1-N$. We conclude that 
$v_{2n+M_1+1}(\alpha)  \stackrel\nearrow= \V_{M_1+1}$.
The result follows.
\end{proof}

Next result is now an immediate consequence of the preceding lemma:

\begin{prop}\label{prop well apprximable}
Let $M\geq N$.
An irrational element $\alpha$ is $M$-approximable if and only if the image in the $\mathbf{S}$-graph $\Gamma \backslash \mathfrak{t}$ of the ray $\mathfrak{r}_\alpha$ going from $\tilde{v}_0$ to $\alpha$
satisfies $v_i(\alpha)=\V_{M+1}$ infinitely often. In particular,
$\alpha$ is well approximable if and only if the corresponding
walk visits  the strict standard ray $\mathfrak{r}_{\infty}^{s}$ infinitely often. \qed
\end{prop}

\paragraph{Proof of Theorem~\ref{T3}.}
Let $\alpha\in\mathcal{O}_\infty$ be an irrational. If 
$\alpha$ has good approximation, 
there is nothing to prove, so we assume this is 
not the case.
Assume that $v_i(\alpha)$ lies outside
the strict standard ray, at distance $r>0$ from the 
vertex  $\V_{N+1}$, the image of $w_0^{[-N-1]}$. Assume also
that neither of the neighboring vertices $v_{i-1}(\alpha)$ 
and $v_{i+1}(\alpha)$ is closer to
$\V_{N+1}$. Then we can choose an irrational $\beta$ satisfying 
$\tilde{v}_i(\beta)=\tilde{v}_i(\alpha)$, and $v_{i+r}(\beta)=\V_{N+1}$. It follows that $\beta$
has a good approximation $\frac fg$ satisfying $\nu_{\infty}\left(\beta-\frac fg\right)>i+r=N+1-2\nu_{\infty}(g)$.
We conclude that $$\nu_{\infty}\left(\alpha-\frac fg\right)\geq i=(N+1-r)-2\nu_{\infty}(g)
> (N-r)-2\nu_{\infty}(g).$$
Note that, for every irrational $\alpha$, there are infinitely many values of $i$ for which the 
value $r-1$ can be bounded by the diameter of the center. The result follows.
\qed

$\phantom{A}$

Recall that we denote by 
$\lbrace \V_i | \,  i \geq N+1 \rbrace $
the vertex set of the strict standard ray 
$\mathfrak{r}_{\infty}^s$. See \S~\ref{subsection quotients} 
for more details. Recall also that $m_{v,v'}$ denotes the neighbor multiplicity. See Def.~\ref{nmul}.

\begin{lem}\label{lemma m < q-1}
If $m_{\V_N,\V_{N+1}}\leq{q^{\deg(P_{\infty})}-1}$ then there exists an element $\alpha \in \mathcal{O}_{\infty}$ that is not well approximable.     
\end{lem}

\begin{proof}
Recall that $\iota(z)=-1/z$.
Define the element $\alpha$ in a way that $v_t(\alpha)=\V_t$ for 
$t\leq N$. This can be done since $\tilde{v}_{-t}=w_0^{[t]}$ 
is also a pre-image of $\V_{t}$, as 
$\mathbf{g}_\iota*\tilde{v}_{-t}=\tilde{v}_{t}$. Now, by the hypotheses, it is possible to choose 
$v_{N+1}(\alpha)$ different from $\V_{N+1}$. If, for any subsequent
value of the index $i$, we have  $v_i(\alpha)=\V_N$, we choose
$v_{i+1}(\alpha)$ different from $\V_{N+1}$.
Other than that, the
element $v_i(\alpha)$ can be chosen arbitrarily. 
Note that the walk cannot enter $\mathfrak{r}_\infty^s$ if not
through $\V_N$, so that any walk satisfying the prescribed
conditions defines an element $\alpha$
that is not well approximable.
The result follows.
\end{proof}

\paragraph{Proof of Theorem~\ref{T2}.}
For the first statement, we reason as in the preceding lemma, noting
that we can choose $v_{i+1}(\alpha)=\V_N$ whenever
$v_i(\alpha)=\V_{N+1}$ (c.f. Proposition~\ref{prop well apprximable}).

Now we prove the second statement.
By Lemma \ref{lemma m < q-1},
we can assume that $m_{\V_N,\V_{N+1}}$ is either $q^{\deg(P_{\infty})}$
or $q^{\deg(P_{\infty})}+1$. If 
$m_{\V_N,\V_{N+1}}=q^{\deg(P_{\infty})}+1$ there are no additional vertices. In this case $N=0$, since $\mathfrak{r}_\infty$ is a ray. 
Otherwise,
there is a unique neighbor of $\V_N$ apart from $\V_{N+1}$. 
Let us call it $u_1$.  We iterate this process. 
If we ever reach a vertex
$u_k$ such that $m_{u_k,u_{k-1}}\leq{q^{\deg(P_{\infty})}-1}$, we finish the
proof as before, otherwise, we either reached the end of the graph,
or we have a unique neighbor of $u_k$ apart from $u_{k-1}$ that
we can call $u_{k+1}$. We conclude that the \textbf{S}-graph $\Gamma \backslash \mathfrak{t}$ is a ray or a maximal path. By uniformity
we set $u_0=\V_N$.

The image $\V_0$ of the vertex $w_0^{[0]}$ must be
equal to $u_N$ since the standard ray is a ray.
 In particular, it has only one
or two neighbors. However, as noted in Lemma~\ref{L62},
the edges starting at $\V_0$ are
in correspondence with the orbits of
$\mathrm{GL}_2(\mathbb{F})$ on $\mathbb{P}^1(\mathbb{F}[P_\infty])$.
It follows that 
$\deg(P_\infty)=\Big[\mathbb{F}[P_\infty]:\mathbb{F}\Big]\leq 3$.
The cases $\deg(P_\infty)\in\{2,3\}$ can be discarded as neither of
the two orbits, namely $\mathbb{P}^1(\mathbb{F})$ and its complement,
has size $1$. It remains the case $\deg(P_\infty)=1$. 
In particular $N=2g-1$. 
If $g\geq1$, then $m_{u_1,u_0}=m_{\V_{2g-1},\V_{2g}}=1\leq
q^{\deg(P_\infty)}-1$,
according to the explicit formula for
the stabilizer of $\tilde{v}_i$ give in
Lemma \ref{lemma stab in gamma}.
The result follows as before.
We conclude that $g=0$ and $\deg(P_{\infty})=1$, which implies that $A \cong \mathbb{F}[x]$. \qed

\section{Random processes and Haar measure}\label{section haar measure}

In this section we use the standard notations 
in probability 
theory. We assume $(\Omega,\mathfrak{B},P)$ is 
a probability space, 
i.e., a set $\Omega$ with a probability  measure
$P:\mathfrak{B}\rightarrow[0,1]$, where
$\mathfrak{B}\subseteq\wp(\Omega)$ is a $\sigma$-algebra. 
The subsets in $\mathfrak{B}$ are called events. 
When $P(A)\neq0$,
for $A\in\Omega$, we write $P(B|A)=P(A\cap B)/P(A)$ for 
the conditional probability. A random element of a 
set $\mathbf{Y}$ is a function
$Y:\Omega\rightarrow\mathbf{Y}$, and we often use 
$P(Y\in \mathbf{T})$ for
the probability of the event 
$\{\omega\in\Omega|Y(\omega)\in \mathbf{T}\}$,
for any subset $\mathbf{T}\subseteq\mathbf{Y}$. 
A random process is a sequence
of random elements $\{Y_n\}_{n\in\mathbb{N}}$. 
A sequence of the form
$\{Y_n(\omega)\}_{n\in\mathbb{N}}$, for $\omega\in\Omega$ 
is called a trajectory. We say that an event $A$ depends 
on $B$, but not on $C$
if $P(A|B\cap C)=P(A|B)$. Similar conventions apply 
to random elements. A Markov chain is a process 
$\{Y_n\}_{n\in\mathbb{N}}$ where
$Y_{n+1}$ depends on $Y_n$, but no on $Y_{n-1}$,
or any other former element.
In this case elements in the set $\mathbf{Y}$ 
are called states of the Markov chain.

The rays described at the beginning of \S~\ref{section walk} can be seen as 
trajectories of a random process. We let $\mathfrak{R}$ denote a random
ray whose $n$-th vertex is $\widetilde{V}_n$. This implies that the vertex 
$\widetilde{V}_0$ is surely
$w_0^{[0]}$, or in symbols 
$P \left(\widetilde{V}_0=w_0^{[0]} \right)=1$. We assume that 
the conditional probabilities satisfy $P \left(\widetilde{V}_{n+1}=w_{n+1}\big|\widetilde{V}_n=w_n \right)=0$, 
 unless the ball $B_{n+1}$ corresponding to $w_{n+1}$
is a maximal
proper sub-ball of the ball $B_n$ corresponding to $w_n$, 
and in the latter case
$P \left(\widetilde{V}_{n+1}=w_{n+1} \big|\widetilde{V}_n=w_n \right)=\frac1q$. Thus defined, 
$m(A)=P\big(\alpha(\mathfrak{R})\in A\big)$ is the Haar measure on Borel subsets of $\mathcal{O}_\infty$, normalized as to have
$m(\mathcal{O}_\infty)=1$. We also use the directed random edge
$\widetilde{E}_n$, with source $s(\widetilde{E}_n)=\widetilde{V}_{n-1}$ and 
target $t(\widetilde{E}_n)=\widetilde{V}_n$, for $n\geq1$.

Now we let $\Theta$  be an arbitrary subgroup of 
$\mathrm{GL}_2(K)$, and let $\phi:\mathfrak{t}\rightarrow \Theta\backslash\mathfrak{t}$
be the canonical projection. Then $V_n=\phi(\widetilde{V}_n)$ is 
a random vertex in $\Theta\backslash\mathfrak{t}$ and  $E_n=\phi(\widetilde{E}_n)$ is a random
directed edge. A quick computation, using the neighbor multiplicities (see Def.~\ref{nmul}) shows that
$$ P\left( V_{n+1}=w\big|V_n=v,V_{n-1}=w' \right)=\left\{
\begin{array}{rcl}
\frac{m_{v,w}}q     &\textnormal{if}& w\neq w'  \\
\frac{m_{v,w}-1}q     &\textnormal{if}& w=w'  \\
\end{array}
\right..
$$
This process is not a Markov Chain when the states are vertices, which is a consequence of the fact that the process in $\mathfrak{t}$ does not allow 
backtracking.  We can fix this by considering directed paths $E_n$ as states. In fact, if $m'_{v,e}$ denote the number of pre-images of an edge
$e$ among the edges with a given source $v$, we have
$$ P\left(E_{n+1}=e \big|E_n=e'\right)=
\left\{
\begin{array}{rcl}
\frac{m'_{t(e'),e}}q   &\textnormal{if}& e\neq r(e')  \\
\frac{m'_{t(e'),e}-1}q     &\textnormal{if}& e=r(e')  \\
\end{array}
\right.,
$$
where $r(e)$ denotes the reverse edge of $e$. As the 
latter probability does not depends on $n$, this process is a 
time-homogeneous Markov chain (THMC). We refer to it as the standard 
THMC. Note that, if $m'_{t(e'),e}>0$, then the corresponding 
probability is also positive, except possibly
in the case $e=r(e')$. This is critical 
in later computations, and we use it throughout 
without further ado.

Recall that a state $e$ in a THMC is called recurrent if, starting from the
state $e$, the probability that the process will eventually return to that state
is $1$. A process returns to a recurrent state infinitely many times with 
probability one. If $e$ is a recurrent state, and if $e'$ is a state that can be
reached from $e$, that is if $P\left(E_{k+s}=e' \big|E_k=e\right)>0$ for some $s>0$,
then $e'$ is also recurrent.

\begin{lem}\label{L51}
If the quotient graph $\Theta\backslash\mathfrak{t}$ is not a cycle, the reverse $r(e)$ of a recurrent state $e$ of the standard THMC is also recurrent.
\end{lem}

\begin{proof}
    Since the directed edge $e$ is recurrent, there is a minimal sequence of
    directed edges $e=e_0e_1e_2\cdots e_n=e$, with $n\geq1$, satisfying
    $P\left(E_{t+1}=e_k+1 \big|E_t=e_k\right)>0$. If $r(e)$
    appears in this sequence, there is nothing to prove, so we assume this 
    is not the case. By minimality, we can also assume that $e$ fails to appear,
    except for the initial and final positions. We claim that the edges in the
    sequence $e=e_0e_1e_2\cdots e_{n-1}$ form a cycle.
    Since the \textbf{S}-graph 
    is bipartite, the integer $n$ is a positive even integer, and the claim is immediate for $n=2$. Assume $n\geq4$.
    If the source or target of $e$ is the source of target of some edge 
    $e_i$ with $2\leq i\leq n-2$, part of the sequence can be deleted,
    again contradicting minimality.  Analogously, if there are two different edges 
    $e_i$ and $e_j$ with the same target, we could delete the sub-sequence
    $e_{i+1}\cdots e_j$ unless $e_{j+1}=r(e_i)$, but in this case also
    $e_{i-1}$ and $e_{j+1}$ have the same target, and we could delete an even bigger
    sub-sequence, a process that cannot be indefinitely iterated, since $r(e)$ does not appear. 
    This prove the claim.

    Since $\Theta\backslash\mathfrak{t}$ is not a cycle, there must exists some
    vertex in the cycle $e_0e_1e_2\cdots e_{n-1}$ having an additional edge that
    can be used to leave the cycle. Since $e$ is recurrent, it must be possible
    to return to the cycle, by the same or a different edge, but then it is possible
    to enter the cycle in the opposite direction. The result follows.
\end{proof}

\begin{rem}
The condition that $\Theta\backslash\mathfrak{t}$ is not a cycle 
is essential in the  preceding result. Assume $\Theta$ is the group
generated by the matrix
    $\left(\begin{array}{cc} \pi^n & 0 \\ 0 & 1\end{array}\right)$ and the set
    $\left(\begin{array}{cc} 1 & F \\ 0 & 1\end{array}\right)$, where $F$ is an additive group satisfying $B_0^{[-n]}+f=B_0^{[-n]}$,
    for each $f\in F$,
    and acting transitively on the set of maximal
    sub-balls $B_0^{[-n]}$. 
    Then the quotient $\Theta\backslash\mathfrak{t}$ is a cycle. Furthermore, 
    for every pair $(v,w)$ of neighboring vertices in this quotient we have
    $\{m_{v,w},m_{w,v}\}=\{1,q\}$, with the value of $m_{v,w}$ depending on the 
    orientation
    of the edge with source $v$ and target $w$. It is not hard to show 
    that the standard THMC, for this group,
    has recurrent edges in one direction and transient edges in the opposite direction.
\end{rem}

\begin{ex}
Now let $\Theta$ be the group
    $\left(\begin{array}{cc} 1 & K \\ 0 & 1\end{array}\right)$,
    of all traslations, then the quotient $\Theta\backslash\mathfrak{t}$ is a
    line, and the standard THMC, for this group, is deterministic, as it is 
    start downwards and  upwards turns are not possible, but even a walk starting upwards will eventually turn downwards with probability 1,
    so every directed edge is transient in this case.
\end{ex}

\begin{lem}\label{L52}
Assume $\Theta=\Gamma=\mathrm{GL}_2(A)$ as before. Then, the standard THMC visit the center of the graph $\Gamma\backslash\mathfrak{t}$ 
    infinitely often with probability $1$.
\end{lem}

\begin{proof}
 It suffices to prove that every time the process leaves the center, it returns
 to it with probability 1. This is a consequence of the structure of cusps. 
 for a cusp with vertices $v_0,v_1,v_2,\dots$, as in Figure~\ref{F2}(A), we have
 $m_{v_k,v_{k+1}}=1$, while $m_{v_k,v_{k+1}}=q$. This means that, as soon as
 the process moving to the right on the cusp gives a step back,  it can no longer
 resume the rightward movement until it reaches the center again. This happens
 with probability one, whence the result follows.
\end{proof}

 \begin{figure}
\[  
\mathbf{(A)}\unitlength 1mm 
\linethickness{0.4pt}
\ifx\plotpoint\undefined\newsavebox{\plotpoint}\fi 
\begin{picture}(40,15)(0,-6)
\put(-5,0){\line(0,1){10}}\put(-5,10){\line(1,0){10}}\put(-5,0){\line(1,0){10}}\put(5,0){\line(0,1){10}}
\put(5,2.5){\makebox(0,0)[cc]{$\bullet$}}\put(5.2,2.5){\line(1,0){8}}\put(13.7,2.5){\makebox(0,0)[cc]{$\bullet$}}
\put(14.2,2.5){\line(1,0){8}}\put(22.7,2.5){\makebox(0,0)[cc]{$\bullet$}}
\put(23.2,2.5){\line(1,0){8}}\put(31.7,2.5){\makebox(0,0)[cc]{$\bullet$}}
\put(36,2.5){\makebox(0,0)[cc]{$\cdots\cdots$}}
\put(0,5){\makebox(0,0)[cc]{${}_Z$}}
\put(7.4,0){\makebox(0,0)[cc]{${}_{v_0}$}}
\put(13.7,5){\makebox(0,0)[cc]{${}_{v_1}$}}
\put(22.7,0){\makebox(0,0)[cc]{${}_{v_2}$}}
\put(31.7,5){\makebox(0,0)[cc]{${}_{v_3}$}}
\end{picture}
\qquad\qquad
\mathbf{(B)}\unitlength 1mm 
\linethickness{0.4pt}
\ifx\plotpoint\undefined\newsavebox{\plotpoint}\fi 
\begin{picture}(40,15)(0,-6)
\put(0,2.5){\makebox(0,0)[cc]{$\bullet$}}\put(0,2.5){\line(1,0){8}}\put(8.7,2.5){\makebox(0,0)[cc]{$\bullet$}}
\put(22.7,2.5){\makebox(0,0)[cc]{$\bullet$}}
\put(23.2,2.5){\line(1,0){8}}\put(31.7,2.5){\makebox(0,0)[cc]{$\bullet$}}
\put(4,2.5){\makebox(0,0)[cc]{$>$}}
\put(27,2.5){\makebox(0,0)[cc]{$>$}}
\put(16,2.5){\makebox(0,0)[cc]{$\cdots\cdots\cdots$}}
\put(0.2,0){\makebox(0,0)[cc]{${}_{s(e)}$}}
\put(11.2,5){\makebox(0,0)[cc]{${}_{t(e)}$}}
\put(22.2,0){\makebox(0,0)[cc]{${}_{s(e')}$}}
\put(31.7,5){\makebox(0,0)[cc]{${}_{t(e')}$}}
\end{picture}
\]
\caption{Diagrams used in the proofs of Lemmas~\ref{L52} and Lemma~\ref{L53}. The box marked with a ``$Z$'' denotes a connected subgraph.}\label{F2}
\end{figure}
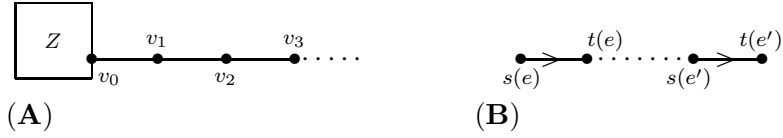

\begin{lem}\label{L53}
Assume $\Theta=\Gamma$. Then, every directed edge of the graph $\Gamma\backslash\mathfrak{t}$ is a recurrent state of the standard THMC.
\end{lem}

\begin{proof}
By Lemma~\ref{L52}, the process visit the center infinitely often with 
probability one. Since the center is a finite graph, it must contain
at least one recurrent edge $e$. Let $e'$ be a different edge. Since 
$\Gamma\backslash\mathfrak{t}$ is connected, there must exists one path $\gamma$
connecting one edge to the other. Since $\Gamma\backslash\mathfrak{t}$ is not
a cycle, as it has cusps, Lemma~\ref{L51} applies and we can assume that the
edges are oriented as in Figure~\ref{F2}(B), but this case is immediate. The result follows.
\end{proof}

We assume that $\Theta=\Gamma$ is all that follows.
\\ 

\paragraph{Proof of Theorem~\ref{T1}}
Fix an integer $M\geq N$. Let $\mathfrak{r}_\infty$ be the standard ray 
 in $\Gamma\backslash\mathfrak{t}$, 
and let $\mathfrak{r}_M$ denote
the smallest sub-ray of $\mathfrak{r}_\infty$ 
containing the vertex $\mathbf{v}_{j}$ for all $j > M$.
Then we have $$m(\Psi_M)=P(\alpha(\mathfrak{R})\in\Psi_M)=
P\big(\sharp\{n| V_n\in\mathfrak{r}_M\}=\infty\big)\geq
P\big(\sharp\{n|E_n=e\}=\infty\big),$$
for any fixed directed edge $e$ inside the ray $\mathfrak{r}_M$.
Since all edges are recurrent by Lemma~\ref{L53}, we conclude that
the probability on the right hand side of the inequality is $1$, and hence so is
$m(\Psi_M)$. For the last statement, we note that every $\alpha\in\Sigma=\bigcap_{M\geq N}\Psi_M$
is the limit of such a sequence $\left\{\frac{f_n}{g_n}\right\}_n$, and the intersection of a numerable family
of full measure sets has full measure.
\qed

\medskip

\paragraph{Proof of Theorem~\ref{T1b}}
Let $\rho>2$ and $M \in \mathbb{Z}$.
Note that, according to our characterization of $M$-approximations in
Lemma~\ref{PWA2}, if we have $\alpha\in\Phi_{M,\rho}$,
then there are sequences $\{n_i\}_i$ and $\{N_i\}_i$ satisfying the following conditions:
\begin{itemize}
    \item[(1)] $\lim_{i \to \infty} n_i= \infty$,
    \item[(2)] $N_i>M+(\rho-2)n_i$ and
    \item[(3)] $v_{n_i+N_i}(\alpha)\stackrel\nearrow= \mathbf{v}_{N_i}$.
\end{itemize}
Note that the last condition makes sense for $N_i>N$, which is 
guaranteed by the first two conditions for large $i$.
Now take a random path $\mathfrak{R}$.
Recall that, for all pairs of positive integers $m$ and $t$,
with $t\geq N$, we can have $V_{n+t}\stackrel\nearrow=\V_t$  only if
$V_{n+N}=\V_N$. We conclude that
$$P\left(V_{n+t}\stackrel\nearrow=\V_t\right)=
P\left(V_{n+t}=\V_t\big|V_{n+N+1}=\V_{N+1}\right)
P\left(V_{n+N+1}=\V_{N+1}\right)$$
$$\leq P\left(V_{n+t}=\V_t\big|V_{n+N+1}=\V_{N+1}\right)=
\left(\frac1{q+1}\right)^{t-N-1}.$$
For any fix $n>0$, let $B_n$ be the event that $V_{n+T}\stackrel\nearrow=\V_T$ 
for some $T>M+(\rho-2)n$.
Then $P(B_n)<\left(\frac1{q+1}\right)^{M+(\rho-2)n-N-1}<\frac1{n^2}$ 
for all sufficiently large integer $n$.
This implies that the series $\sum_nP(B_n)$ is convergent. 
Note that the event $B_\infty=\bigcap_{n=1}^{\infty}\bigcup_{m\geq n}B_m$
can be re-written as follows:
$$B_\infty=\{\omega|V_{n+T}(\omega)=\V_{n+T},\textnormal{ for some }T\geq M+(\rho-2)n,
\textnormal{for infinitely many values of }n\}.$$
Then
$$m(\Phi_{M,\rho})=P(B_{\infty})=
P\left(\bigcap_{n=1}^{\infty}\bigcup_{m\geq n}B_m\right)\leq
\lim_{n\rightarrow\infty}\sum_{m\geq n}P(B_m)=0.$$
\qed

\section{Examples}\label{section examples}

\begin{ex}
When $A=\mathbb{F}[x]$, it is well known that every element 
$\alpha\in K_\infty$ has a continued fraction expansion of the form
\begin{equation}\label{eq7a}
\alpha=f_0+\frac1{f_1+\frac1{f_2+\frac1{f_3+\dots}}},
\end{equation}
where every $f_i$ is a polynomial of positive degree.
Note that $f_0=0$ precisely when $\alpha\in\mathcal{O}_\infty$.
For a random element $\alpha=\alpha(\mathfrak{R})$, as before,
the remaining coefficients $f_i$ can be shown to be independent and 
identically distributed random elements. 

In this case, the quotient
graph $\Gamma\backslash\mathfrak{t}$ is a ray with vertices
$\V_0,\V_1,\V_2,\dots$, without any additional vertex or edge.
The sequence $v_0(\alpha),v_1(\alpha),v_2(\alpha),\dots$ of 
vertices in the walk of an irrational element $\alpha$
goes up in that ray, reaching a local maximum  
$\V_{n_1}$, then heading back to $\V_0$, then going up again
reaching a second local maximum $\V_{n_2}$, and iterating that
behaviour. The indices of these successive local maxima
are the degrees of the polynomials $f_1,f_2,f_3,\dots$, as F. Paulin pointed out in \cite{Paulin}.
Then, the Law of Large numbers can be used to show that
every polynomial appears infinitely many times in the
continued fraction with probability one, thus recovering
the fact that all edges are recurrent in this case.
\end{ex}

\begin{ex}\label{E2}
For a general $A$ as in \S~\ref{subsection quotients}, we can find a set
of representatives  $A'\supseteq A$ of $K_\infty/\pi\mathcal{O}_\infty$, 
where $\pi \in \mathcal{O}_{\infty}$ is a uniformizing parameter of 
$K_{\infty}$. Then every element 
$\alpha\in K_\infty$ has a unique expression as an infinite fraction
\begin{equation}\label{eq7}
\alpha=b_0+\frac1{b_1+\frac1{b_2+\frac1{b_3+\dots}}},
\end{equation}
where $b_i\in A'$, while $\mathrm{deg}(b_i)>0$ for $i>0$. 
If $A\neq A'$, reasoning via random walks, we can easily prove
that the set of elements for which we have a continued fraction expression 
of the form \eqref{eq7}, with all its coefficients in $A$,
is a set of zero Haar measure.
\end{ex}

\begin{ex}\label{E3}
Assume that $g(X)=0$ and $\deg(P_{\infty})=2$.
Then, the $\mathbf{S}$-graph $\Gamma \backslash \mathfrak{t}$ is a maximal path as in Example~\ref{ex S and C graph in degree two}.
See Figure~\ref{Figure degree two}(A).
In particular, we can lift $\Gamma \backslash \mathfrak{t}$ to a subtree $\mathfrak{p}$ of $\mathfrak{t}$.
It follows from \cite[Theorem 3.4]{Mason0} that $\mathfrak{p}$ 
can be chosen as the maximal path $\mathfrak{p}[x,\infty]\subset\mathfrak{t}$ joining the visual limits $x$ and $\infty$.
We write $V(\mathfrak{p})=\lbrace v_i \rbrace_{i=0}^{\infty} \cup \lbrace v^{*}_i \rbrace_{i=1}^{\infty}$, where $v_i= w^{[-i]}_0$ and $v^*_i=w_x^{[i]}$.
The image of $v_i$ is $\V_i$, while we use $\V_i^*$ for the image of
$v^*_i$. 

It follows from Example~\ref{ex S and C graph in degree two} that there is a global matrix $\mathbf{m}$ interchanging both cusps, 
sending $v_i$ to $v_{i-1}^*$. In particular, it suffices to compute the
initial vertex of the cusp containing the strict standard ray (c.f. Lemma~\ref{lemma strict is cuspidal}).
The strict 
standard ray has the initial vertex $\V_{N+1}=\V_1$, which is the image of 
$v_1$. The initial vertex of such cusp is $\V_0$. In fact, the cusp cannot 
extend beyond this point, since
the action of $\mathrm{GL}_2(\mathbb{F})$ on the neighbors of $v_0$ is
the same as in the action of the projective line
$\mathbb{P}^1\Big(\mathbb{F}[P_\infty]\Big)$. It follows that there are
two orbits of size $q+1$ and $q^2-q=q(q-1)$. In particular, both $\V_0$
and $\V^*_1$ are initial vertices of cusps. So the diameter of the center of the $\mathbf{S}$-graph
is one in this case.
We conclude from Theorem~\ref{T3} that any element in $\mathcal{O}_{\infty}$ is $(-2)$-approximable.

Since $A \neq \mathbb{F}[x]$, there exists a non $0$-approximable element in $\mathcal{O}_{\infty}$, according to Theorem~\ref{T2}.
Since the stabilizer of $v_1^*$ acts on the neighbors with two orbits, neither
of which has size 1, the same argument use in the proof of Theorem~\ref{T2}
can be applied to $\V_1^*$, instead of $\V_0$, to prove the existence
of elements that are not $(-1)$-approximable. 
Namely we can choose $\alpha$ in a way that $v_n(\alpha)=\V_1^*$, for some $n$,
choosing $v_{n+1}(\alpha)=\V_2^*$, and from there on, every time that the walk
reaches $\V_1^*$, choosing $\V_2^*$ in the following step.
We conclude that the bound $-2$
is optimal.
\end{ex}

\begin{ex}\label{ex 4}
In the setting of Example~\ref{E3},
\begin{figure}
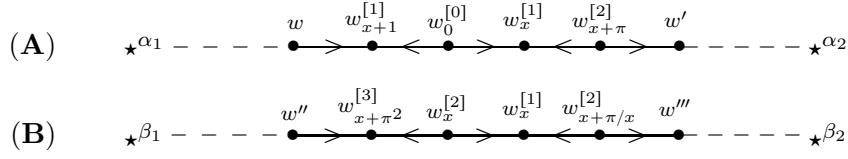

\[
\mathbf{(A)}\qquad
\xygraph{!{<0cm,0cm>;<1cm,0cm>:<0cm,1cm>::}
!{(-2,0) }*+{\star^{\alpha_1}}="vl1"
!{(7,0) }*+{\star^{\alpha_2}}="vl2"
!{(-.2,0) }*+{}="v0"
!{(-0.04,-0) }*+{\bullet}="v1"
!{(1,0) }*+{\bullet}="v2"
!{(2,0) }*+{\bullet}="v3"
!{(3,0) }*+{\bullet}="v4"
!{(4,0) }*+{\bullet}="v5"
!{(5.04,0) }*+{\bullet}="v6"
!{(0,0.30) }*+{{}_{w}}="v1i"
!{(1,0.36) }*+{{}^{w_{x+1}^{[1]}}}="v2i"
!{(2,0.36) }*+{{}_{w_0^{[0]}}}="v3i"
!{(3,0.36) }*+{{}^{w_x^{[1]}}}="v4i"
!{(4,0.36) }*+{{}_{w_{x+\pi}^{[2]}}}="v5i"
!{(5,0.30) }*+{{}^{w'}}="v6i"
!{(5.2,0) }*+{}="v7"
"vl1"-@{--}"v1" "v0"-"v7" "vl2"-@{--}"v6"
!{(0.5,0) }*+{>}!{(1.5,0) }*+{<}
!{(2.5,0) }*+{>}!{(3.5,0) }*+{<}
!{(4.5,0) }*+{>}
}
\]
\[
\mathbf{(B)}\qquad
\xygraph{!{<0cm,0cm>;<1cm,0cm>:<0cm,1cm>::}
!{(-2,0) }*+{\star^{\beta_1}}="vl1"
!{(7,0) }*+{\star^{\beta_2}}="vl2"
!{(-.2,0) }*+{}="v0"
!{(-0.04,-0) }*+{\bullet}="v1"
!{(1,0) }*+{\bullet}="v2"
!{(2,0) }*+{\bullet}="v3"
!{(3,0) }*+{\bullet}="v4"
!{(4,0) }*+{\bullet}="v5"
!{(5.04,0) }*+{\bullet}="v6"
!{(0,0.30) }*+{{}_{w''}}="v1i"
!{(1,0.36) }*+{{}^{w_{x+\pi^2}^{[3]}}}="v2i"
!{(2,0.36) }*+{{}_{w_x^{[2]}}}="v3i"
!{(3,0.36) }*+{{}^{w_x^{[1]}}}="v4i"
!{(4,0.36) }*+{{}_{w_{x+\pi/x}^{[2]}}}="v5i"
!{(5,0.30) }*+{{}^{w'''}}="v6i"
!{(5.2,0) }*+{}="v7"
"vl1"-@{--}"v1" "v0"-"v7" "vl2"-@{--}"v6" 
!{(0.5,0) }*+{>}!{(1.5,0) }*+{<}
!{(2.5,0) }*+{>}!{(3.5,0) }*+{<}
!{(4.5,0) }*+{>}
}
\]
\caption{Finding explicit elements with a prescribed path.}\label{F3}
\end{figure}
 we can explicitly compute elements of $\mathcal{O}_\infty$
that are not well approximable. To fix ideas, set $q=2$. 
Note that the matrix $\mathbf{g}_\iota=\sbmattrix0{-1}10$ is an order-two 
element in the 
stabilizer of $w_0^{[0]}$ that permutes the neighbors $w_x^{[1]}$ and $w_{x+1}^{[1]}$. Similarly,
$\mathbf{g}_\eta=\frac1\pi\sbmattrix{x^2}1{x^2+1}{x^2}$, where $\pi=x^2+x+1$ is the only quadratic 
irreducible polynomial,
 is an order two element in the stabilizer of $w_x^{[1]}$ that permutes the neighbors $w_0^{[0]}$ 
 and $w_{x+\pi}^{[2]}$. 
 This takes us to the situation of Figure~\ref{F3}(A). Note
 that all oriented edges in this picture belong to the same orbit.
 Since the product of two consecutive flips on a line is a shift by two, the visual limits 
 $\alpha_1$ and  $\alpha_2$ of such line, both having a path that alternates between the images 
 of $w_0^{[0]}$ and $w_x^{[1]}$, are the fixed points of the composition 
 $\eta\circ\iota(z)=\frac{x^2+z}{(x^2+1)+x^2z}$. This are given by the equation $\alpha^2+\alpha+1=0$,
i.e., they are the constant of the completion that don't belong to the base field. 
Therefore, these elements are $(-1)$-approximable, but not $0$-approximable.
\end{ex}

\begin{rem}
In the preceding example, $\eta$ is computed from the formula 
$\Gamma_{v^*_1}=  \mathbf{h}^{-1} 
\mathrm{SL}_2(\mathbb{F}) \mathbf{h}$,
 from \cite[2.20]{Mason0}. Here
$\mathbf{h}=\sbmattrix{0}{-\pi}{1}{-x}\in \mathrm{GL}_2(K)$ is a matrix satisfying $\mathbf{h} * v_1^{*}=v_0$.    
\end{rem}

\begin{ex}
In the same setting as in the preceding example, we have from \cite[Pag. 58]{MasonSchweizer2} the following formula:
\begin{equation}\label{eq stab2}
\Gamma_{v^*_2}= \left \lbrace \mathbf{h}_{x}^{-1} \sbmattrix{\alpha}{c \pi^{-2}}{0}{\beta} \mathbf{h}_x \, \bigg| 
\begin{array}{lr}
      \alpha, \beta \in \mathbb{F}^{*}, \, \, c \in \mathbb{F}[x] \\
        \deg(c-(\alpha^{-1}-\alpha)x^3) \leq 2 
    \end{array}
\right \rbrace,
\end{equation}
where $\mathbf{h}_x:=\sbmattrix{0}{-1}{1}{-x}$. We choose $\alpha=\beta=c=1$, and write $\mathbf{g}_\xi=\mathbf{h}_{x}^{-1} \sbmattrix1{ \pi^{-2}}01 \mathbf{h}_x$. Then $\xi(z)=\frac{(x\pi^{-2}+1)z+x^2\pi^{-2}}{\pi^{-2}z+x\pi^{-2}+1}$. Using the fact that $\eta(x)=x+\pi/x$ and
$\xi(\infty)=x+\pi^2$, we obtain the path depicted in Figure~\ref{F3}(B).
The endpoints of this path are the invariant points of the Moebius transformation
$$\eta\circ\xi(z)=\frac{[x^2(x+\pi^2)+1]z+[x^4+x+\pi^2]}{[(x+\pi^2)(x^2+1)+x^2]z+[x^2(x^2+1)+x^2(x+\pi^2)]}.$$ It follows that $\beta_1,\beta_2$ can be found as the roots (in $z$) 
of the quadratic equation
$$[(x+\pi^2)(x^2+1)+x^2]z^2+\pi^2z+[x^4+x+\pi^2]=0.$$
The path of either elements go back and forth between $\V_1^*$ and
$\V_2^*$ indefinitely, so these elements are not $(-1)$-approximable. 
\end{ex}

\begin{ex}\label{E4}
Assume that $g(X)=0$ and $\deg(P_{\infty})=3$.
Then, it can be seen in \cite[Ch. II, \S 2.4.2 (b)]{SerreTrees} 
that the $\mathbf{S}$-graph $\Gamma \backslash \mathfrak{t}$ is the union 
of $2$ maximal path with a common ray as intersection.
In particular, we can lift $\Gamma \backslash \mathfrak{t}$ to a subtree
$\mathfrak{q}$ of $\mathfrak{t}$.
It follows from \cite[Theorem 3.4]{Mason0} that the vertex set of 
$\mathfrak{q}$ is $V(\mathfrak{q})=\lbrace v_i 
\rbrace_{i=0}^{\infty} \cup \lbrace v^{*}_i \rbrace_{i=1}^{\infty} 
\cup \lbrace v^{**}_i \rbrace_{i=1}^{\infty}$, where $v_i= w^{[-i]}_0$, $
v^*_i=w_x^{[i]}$ and $v^{**}_i=w_{x^2}^{[i]}$.
The image of $v_i$ is $\V_i$, while we use $\V_i^*$ (resp. $\V_i^{**}$) for 
the image of $v^*_i$ (resp. $v_i^{**}$) in the $\mathbf{S}$-graph.
See Figure~\ref{Figure degree 3}.

\begin{figure}[ht]
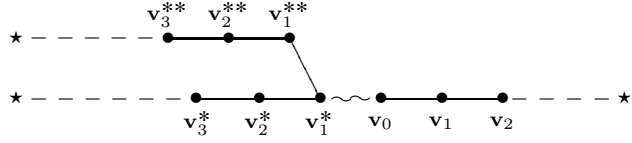

\[
\xygraph{!{<0cm,0cm>;<0.8cm,0cm>:<0cm,0.8cm>::}
!{(-3,0) }*+{\star}="vl1"
!{(7,0) }*+{\star}="vl2"
!{(-3,1) }*+{\star}="vl3"
!{(-0.04,-0) }*+{\bullet}="v1"
!{(1,0) }*+{\bullet}="v2"
!{(2,0) }*+{\bullet}="v3"
!{(3,0) }*+{\bullet}="v4"
!{(4,0) }*+{\bullet}="v5"
!{(5,0) }*+{\bullet}="v55"
!{(-0.5,1) }*+{\bullet}="v6"
!{(0.5,1) }*+{\bullet}="v7"
!{(1.5,1) }*+{\bullet}="v8"
!{(-.2,0) }*+{}="v0"
!{(5.2,0) }*+{}="v5f"
!{(2.1,-0.2) }*+{}="v3f"
!{(2.1,0) }*+{}="v3ff"
!{(2.9,0) }*+{}="v0ff"
!{(1.9,0) }*+{}="v3fff"
!{(3.1,0) }*+{}="v0fff"
!{(1.4,1.2) }*+{}="v8f"
!{(1.6,1) }*+{}="v8ff"
!{(-0.6,1) }*+{}="v6f"
!{(0,-0.4) }*+{{}_{\V_3^*}}="v1i"
!{(1,-0.4) }*+{{}_{\V_2^*}}="v2i"
!{(2,-0.4) }*+{{}_{\V_1^*}}="v3i"
!{(3,-0.4) }*+{{}_{\V_0}}="v4i"
!{(4,-0.4) }*+{{}_{\V_1}}="v5i"
!{(5,-0.4) }*+{{}_{\V_2}}="v55i"
!{(-0.5,1.4) }*+{{}_{\V_3^{**}}}="v6i"
!{(0.5,1.4) }*+{{}_{\V_2^{**}}}="v7i"
!{(1.5,1.4) }*+{{}_{\V_1^{**}}}="v8i"
!{(6.2,0) }*+{}="v7"
"vl1"-@{--}"v1" "v0"-@{-}"v3ff" "v0ff"-@{-}"v5f" "v0fff"-@{~}"v3fff" 
"vl2"-@{--}"v5" 
"v3f"-@{-}"v8f" "v8ff"-"v6f" "v6"-@{--}"vl3"
}
\]
\caption{
The $\mathbf{S}$-graph $\Gamma \backslash \mathfrak{t}$ when $\mathbb{P}^1_{\mathbb{F}}$ and $\deg(P_{\infty})=3$.
The edge represented by a wiggly line equals the central graph of $\Gamma \backslash \mathfrak{t}$.
}\label{Figure degree 3}
\end{figure}

As stated in Lemma~\ref{lemma strict is cuspidal}, the vertices 
in $\lbrace \V_i \rbrace_{i=1}^{\infty}$ are contained in the strict
standard ray, which is a cusp.
Moreover, it follows from the method exposed in 
\cite[Theorem 1.5]{ArenasQuotient}, or alternatively from 
\cite[Lemma 2.17]{Mason0}, that
the vertices in $\lbrace v^{*}_i \rbrace_{i=2}^{\infty}$ (resp. in $ \lbrace v^{**}_i \rbrace_{i=1}^{\infty}$) are contained in other cusps.
Thus, the diameter of the central graph of $\Gamma \backslash \mathfrak{t}$ is $1$.
Therefore, Theorem~\ref{T3} implies that any element in $\mathcal{O}_{\infty}$ is $(-2)$-approximable.
\end{ex}

\begin{ex}\label{E5}
Now, assume that $g(X)=1$. 
In other words, assume that $X$ is a projective elliptic curve over $\mathbb{F}$ defined by a Weierstrass equation $F(x,y)=0$, where 
$$F(x,y)=y^2+ a_1 xy + a_3 y - x^3- a_2 x^2 - a^4 x - a_6, \quad a_i \in \mathbb{F}.$$
Let us also assume that $\deg(P_{\infty})=1$. For instance, we can take $P_{\infty}$ as the only point of $X$ in the line at infinity of $\mathbb{P}^2_{\mathbb{F}}$.
Then, we have that $N=1$.
It follows from \cite[Theorem 3 \& Theorem 5]{Takahashi} that the diameter of the $\mathbf{S}$-graph $\Gamma \backslash \mathfrak{t}$ is:
\begin{itemize}
    \item[(1)] $d=4$, when $F(x,y)$ does not have rational solutions, .i.e., $X(\mathbb{F}) = \emptyset$,
    \item[(2)] $d=6$, in any other case. 
\end{itemize}
Thus, any element in $\mathcal{O}_{\infty}$ is $(-4)$-approximable whenever $X(\mathbb{F}) = \emptyset$, while any element in $\mathcal{O}_{\infty}$ is $(-6)$-approximable whenever $X(\mathbb{F}) \neq \emptyset$.
In both cases, for each $M \geq 1$, Theorem~\ref{T2}
implies the existence of elements in $\mathcal{O}_{\infty}$ which are not $M$-approximable.

\begin{figure}[ht]
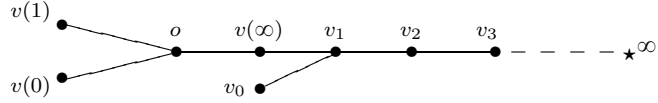

\[
\xygraph{!{<0cm,0cm>;<1cm,0cm>:<0cm,1cm>::}
!{(8,0) }*+{\star^{\infty}}="vl2"
!{(-.2,0) }*+{}="v0"
!{(0,0.5) }*+{{}^{v(1)}}="v1"
!{(0,-0.5) }*+{{}^{v(0)}}="v2"
!{(0.35,0.38) }*+{}="v1v"
!{(0.35,-0.38) }*+{}="v2v"
!{(1.95,0) }*+{}="v0v"
!{(1.8,0) }*+{}="v0"
!{(3,0.0) }*+{\bullet}="v4i"
!{(4,0) }*+{\bullet}="v5i"
!{(3,-0.5) }*+{\bullet}="v5bi"
!{(5,0) }*+{\bullet}="v6i"
!{(6.1,0) }*+{\bullet}="v6i"
!{(1.9,0) }*+{\bullet}="v0i"
!{(6.2,0) }*+{}="v7i"
!{(0.4,0.35) }*+{\bullet}="v1i"
!{(0.4,-0.35) }*+{\bullet}="v2i"
!{(1.9,0.2) }*+{{}^{o}}="v3i"
!{(3,0.25) }*+{{}^{v(\infty)}}="v4"
!{(4,0.2) }*+{{}^{v_1}}="v5"
!{(2.68,-0.5) }*+{{}_{v_0}}="v5b"
!{(4.15,0.05) }*+{}="v5v"
!{(2.96,-0.5) }*+{}="v5bv"
!{(5,0.2) }*+{{}^{v_2}}="v6"
!{(6,0.2) }*+{{}^{v_3}}="v6"
!{(6.2,0) }*+{}="v7"
"v0"-"v7" "vl2"-@{--}"v7i" "v0v"-"v1v" "v0v"-"v2v" "v5v"-"v5bv" 
}
\]
\caption{A lift in $\mathfrak{t}$ of the $\mathbf{S}$-graph when $\mathcal{C}$ is an elliptic curve without rational solutions and $\deg(P_{\infty})=1$.
In the image we assume $q=2$.}\label{F4}
\end{figure}

Now, set $\mathbb{F}=\mathbb{F}_2$ and $F(x,y)=y^2+y+x^3+x+1$, so that $X(\mathbb{F})=\emptyset$.
The $\mathbf{S}$-graph $\Gamma \backslash \mathfrak{t}$ is a tree with only one cusp, according to \cite{Takahashi} (or alternatively \cite{MasonSchweizer1} or \cite[Ch. II, Ex. 2.2.4]{SerreTrees}).
We lift $\Gamma \backslash \mathfrak{t}$ to a subtree $\mathfrak{t}_0$ of $\mathfrak{t}$ such that $\mathrm{V}(\mathfrak{t}_0)=\lbrace v_i \rbrace_{i=0}^{\infty} \cup \lbrace o, v(0), v(1), v(\infty) \rbrace$, where $v_i= w_0^{[-i]}$, $v(\infty)=w_{1/\pi}^{[0]}$, $o=w_{1/\pi}^{[1]}$, $v(0)=w_{1/\pi}^{[2]}$ and $v(1)=w_{1/\pi+\pi}^{[1]}$. See Figure~\ref{F4}.
It follows from \cite[Prop. 9]{Takahashi} that for $n \geq 1$:
\begin{equation}\label{eq stab g=11}
 \Gamma_{v_n}=  \sbmattrix{1}{A[n]}{0}{1}.
\end{equation}
while
\begin{equation}\label{eq stab g=12}
\Gamma_{o}=\lbrace \mathrm{id} \rbrace, \quad \Gamma_{v_0}= \mathrm{GL}_2(\mathbb{F}), \quad \Gamma_{v(\infty)}=  \sbmattrix{1}{\mathbb{F}}{0}{1}. 
\end{equation}
The image of $v_i$ in the $\mathbf{S}$-graph is $\V_i$.
We respectively denote by $\mathbf{o}$, $\mathbf{v}(0)$, $\mathbf{v}(1)$ and $\mathbf{v}(\infty)$ the images of $o$, $v(0)$, $v(1)$ and $v(\infty)$ in the $\mathbf{S}$-graph. 
Lemma~\ref{lemma standard ray} imply that, for each 
$i \geq 2$, we have $m_{\V_i, \V_{i+1}} =1$ and 
$m_{\V_i, \V_{i-1}} =2.$
The first equality in Eq.~\eqref{eq stab g=12} shows that 
$m_{\mathbf{o}, \mathbf{w}}=1$, for each $\mathbf{w} \in 
\lbrace \mathbf{v}(0), \mathbf{v}(1) , \mathbf{v}(\infty) 
\rbrace$.
Since $\V_0$ has a unique neighbor, we conclude that
$m_{\V_0, \V_{1}} =3$.
Since the matrix $\sbmattrix{1}{1}{0}{1}$ fixes
$v_1$, while it transposes the remaining two
neighbors of $v(\infty)$, the third equality in 
Eq. \eqref{eq stab g=12} implies that 
$m_{\mathbf{v}(\infty), \V_{1}} =1$ and 
$m_{\mathbf{v}(\infty), \mathbf{o}} =2$.
Finally, since $\mathbf{v}(0)$ and $\mathbf{v}(1)$ 
are terminal, we have that $m_{\mathbf{v}(0), 
\mathbf{o}}=m_{\mathbf{v}(1), \mathbf{o}}=3$.
Note that we can have walks going from the image 
$\V(0)$ of $v(0)$
to the image $\V(1)$ of $v(1)$ and back, 
passing only through 
the image $\mathbf{o}$ of $o$. This proves the existence of irrationals
that are $(-2)$-approximable, but not $(-1)$-approximable. By inspection,
the bound $-2$ is actually sharp.
\end{ex}

\section*{Acknowledgements}
The second author was partially supported by Anid-Conicyt,
through the Postdoctoral fellowship No $74220027$.


\bibliographystyle{amsalpha}
\bibliography{refs.bib}

\providecommand{\bysame}{\leavevmode\hbox to3em{\hrulefill}\thinspace}
\providecommand{\MR}{\relax\ifhmode\unskip\space\fi MR }
\providecommand{\MRhref}[2]{%
  \href{http://www.ams.org/mathscinet-getitem?mr=#1}{#2}
}
\providecommand{\href}[2]{#2}
\begin{thebibliography}{BAMG22}

\bibitem[AACC18]{ArenasArenasContreras}
Manuel Arenas, Luis Arenas-Carmona, and Jaime Contreras, \emph{On optimal
  embeddings and trees}, Journal of Number Theory \textbf{193} (2018), 91--117.

\bibitem[AC14]{ArenasQuotient}
Luis Arenas-Carmona, \emph{Computing quaternion quotient graphs via
  representations of orders}, Journal of Algebra \textbf{402} (2014), 258--279.

\bibitem[ACB22]{ArenasBravo}
Luis Arenas-Carmona and Claudio Bravo, \emph{On genera containing non-split
  {Eichler} orders over function fields}, Journal de Th\'eo. des Nombres de
  Bordeaux \textbf{34} (2022), no.~3, 647--677.

\bibitem[BAM23]{Baier}
Stephan Baier and Esrafil Ali~Molla, \emph{Diophantine approximation with prime
  denominator in real quadratic function fields}, Finite Fields Appl.
  \textbf{91} (2023), Paper No. 102242, 40.

\bibitem[BAMG22]{Baier2}
Stephan Baier, Esrafil Ali~Molla, and Arijit Ganguly, \emph{Diophantine
  approximation with prime restriction in function fields}, Journal of Number
  Theory \textbf{241} (2022), 57--90.

\bibitem[Bra22]{CBHecke}
Claudio Bravo, \emph{Quotients of the {B}ruhat-{T}its tree by function field
  analogs of the {H}ecke congruence subgroups}, 2022, available at
  \url{https://arxiv.org/abs/2205.07328}.

\bibitem[BT72]{BT1}
François Bruhat and Jacques Tits, \emph{Groupes r\'eductifs sur un corps
  local}, Institut des Hautes \'Etudes Scientifiques. Publications
  Math\'ematiques \textbf{41} (1972), 5--251.

\bibitem[GG17]{Ganguly}
Arijit Ganguly and Anish Ghosh, \emph{Dirichlet's theorem in function fields},
  Canad. J. Math. \textbf{69} (2017), no.~3, 532--547.

\bibitem[{Mas}01]{Mason}
A.~W. {Mason}, \emph{{Serre's generalization of Nagao's theorem: an elementary
  approach}}, {Trans. Am. Math. Soc.} \textbf{353} (2001), no.~2, 749--767.

\bibitem[Mas03]{Mason0}
A.~W. Mason, \emph{The generalization of {N}agao's theorem to other subrings of
  the rational function field}, Communications in Algebra \textbf{31} (2003),
  no.~11, 5199--5242.

\bibitem[MS03]{MasonSchweizer1}
AW~Mason and Andreas Schweizer, \emph{The minimum index of a non-congruence
  subgroup of $\mathrm{SL}_2$ over an arithmetic domain}, Israel Journal of
  Mathematics \textbf{133} (2003), 29--44.

\bibitem[MS05]{MasonSchweizer2}
\bysame, \emph{The minimum index of a non-congruence subgroup of
  $\mathrm{SL}_2$ over an arithmetic domain {II}: The rank zero cases}, Journal
  of the London Mathematical Society \textbf{71} (2005), no.~1, 53–68.

\bibitem[Pau02]{Paulin}
Frédéric Paulin, \emph{Groupe modulaire, fractions continues et approximation
  diophantienne en caractéristique p}, Geometriae Dedicata \textbf{95} (2002),
  65--85.

\bibitem[Ser03]{SerreTrees}
Jean-Pierre Serre, \emph{Trees. {Transl}. from the {French} by {John}
  {Stillwell}.}, corrected 2nd printing of the 1980 original ed., Springer
  Monogr. Math., Berlin: Springer, 2003 (English).

\bibitem[{Tak}93]{Takahashi}
Shuzo {Takahashi}, \emph{{The fundamental domain of the tree of $GL(2)$ over
  the function field of an elliptic curve}}, Duke Mathematical Journal
  \textbf{72} (1993), no.~1, 85 -- 97.

\end{thebibliography}

\end{document}